\documentclass[10.7pt]{article}
\usepackage{latexsym}
\usepackage{amsmath, amsthm, amsfonts}
\usepackage{graphics}
\usepackage{graphicx}
\usepackage{color}
\usepackage[table,xcdraw]{xcolor}
\usepackage{float}
\usepackage{array}
\usepackage{hyperref}

\newcommand{\abs}[1]{\left\lvert#1\right\rvert}
\newtheorem{theorem}{Theorem}[section]

\newtheorem{conjecture}{Conjecture}[section]

\newtheorem{example}{Example}[section]
\newtheorem{remark}{Remark}[section]
\def\smallskip{\addvspace{\smallskipamount}}
\def\medskip{\addvspace{\medskipamount}}
\def\bigskip{\addvspace{\bigskipamount}}
\def\makefootline{\baselineskip=24pt \line{\the\footline}}
\def\pagecontents{\ifvoid\topins\else\unvbox\topins\fi
   \dimen@=\dp255 \unvbox255
   \ifvoid\footins\else
      \vskip\skip\footins \footnoterule \unvbox\footins\fi
     \ifr@ggedbottom \kern-\dimen@ \vfil \fi}
\def\footnoterule{\kern-3pt\hrule width 2truein \kern 2.6pt}

\begin{document}

\title{Computational Complex Dynamics of $\displaystyle{f_{\alpha, \beta, \gamma, \delta}(z)=\frac{\alpha z + \beta}{\gamma z^2 +\delta z}}$}

\author{Sk. Sarif Hassan\\
  \small {\emph{Department of Mathematics}}\\
\small {\emph{College of Engineering Studies}} \\
  \small {\emph{University of Petroleum and Energy Studies}}\\
  \small {\emph{Bidholi, Dehradun, India}}\\
  \small Email: {\texttt{\textcolor[rgb]{0.00,0.07,1.00}{s.hassan@ddn.upes.ac.in}}}\\
}

\maketitle
\begin{abstract}
\noindent The dynamics of the family of maps $\displaystyle{f_{\alpha, \beta, \gamma, \delta}(z)=\frac{\alpha z + \beta}{\gamma z^2 +\delta z}}$ in complex plane is investigated computationally. This dynamical system $z_{n+1}=f_{\alpha, \beta, \gamma, \delta}(z_n)=\frac{\alpha z_n + \beta}{\gamma z_n^2 +\delta z_n}$ has periodic solutions with higher periods which was absent in the real line scenario. It is also found that there are chaotic fractal and non-fractal like solutions of the dynamical systems. A few special cases of parameters are also have been taken care.   \\
\end{abstract}

\begin{flushleft}\footnotesize
{\textbf{Keywords:} Family of maps, Local asymptotic stability, Periodicity, Chaotic and Fractal-like trajectory. \\

}.
\end{flushleft}

\section{Brief Review and Results}

Consider the family of rational maps $f_{\alpha, \beta, \gamma, \delta}:\mathbb{C} \rightarrow \mathbb{C}$  of four complex parameters $\alpha$, $\beta$, $\gamma$ and $\delta$ as

\begin{equation}
\displaystyle{f_{\alpha, \beta, \gamma, \delta}(z)=\frac{\alpha z + \beta}{\gamma z^2 +\delta z}}
\label{equation:total-equationA}
\end{equation}%

\noindent
The corresponding discrete dynamical system is
\begin{equation}
\displaystyle{z_{n+1}=f_{\alpha, \beta, \gamma, \delta}(z_n)=\frac{\alpha z_n + \beta}{\gamma z_n^2 +\delta z_n}}
\label{equation:total-equationB}
\end{equation}%

\addvspace{\bigskipamount}
\noindent
The dynamics of this family of maps Eq.(\ref{equation:total-equationA}) has been studied in \cite{D-C-N} when the parameters and initial condition are non-negative real numbers. Also there are many works on first, second and third order rational difference equations are done in \cite{S-H}, \cite {C-L} \cite{Ca-Ch-L-Q-1}, \cite{R-A} and \cite{K-L}. The main results obtained in the real line of the dynamical system Eq.(\ref{equation:total-equationB}) are given below:\\

\begin{itemize}
  \item $\maltese$ Let $\alpha<\gamma$ and $\beta>\delta$, then the following happens \dots

  \subitem $\spadesuit$ There are infinitely many solutions, such that for each one of its subsequence $(z_{2n})_n$ and $(z_{2n-1})_n$, converges to zero and other diverges to infinity.\\
  \subitem $\spadesuit$ There exist solutions which \dots

     \subsubitem $\clubsuit$ converges to zero if $\gamma+\delta>\alpha+\beta$;
    \subsubitem $\clubsuit$ diverges to infinity if $\gamma+\delta<\alpha+\beta$;
    \subsubitem $\clubsuit$ are constant if $\gamma+\delta=\alpha+\beta$;

  \item $\maltese$ Let $\alpha=\gamma$ and $\beta>\delta$ then for each positive solution $(z_n)_n$, one of the subsequences $(z_{2n})_n$ and $(z_{2n-1})_n$ diverges to infinity and other to a positive number that can be arbitrarily large depending on initial values. Further there, are positive initial values for which the corresponding solution increases monotonically to infinity.
  \item $\maltese$ Let $\alpha<\gamma$ and $\beta=\delta$ then for each positive solution $(z_n)_n$, one of the subsequences $(z_{2n})_n$ and $(z_{2n-1})_n$ diverges to zero and other to a nonnegative number. Further there, are positive initial values for which the corresponding solution decreases monotonically to zero.
\end{itemize}

\noindent
Let us make a note that in the real line set-up, there exists neither any periodic nor chaotic solutions of the dynamical system of the family of maps. We shall now see if such kind of solutions exist of the same dynamical system Eq.(\ref{equation:total-equationB}) in the complex plane.\\\\
\noindent
Here our main aim to investigate computationally the dynamics of the family of maps (\ref{equation:total-equationA}) under the assumption that the parameters and initial condition are arbitrary complex numbers. Similar works on rational maps are made in \cite{S-E} and \cite{S-S1}.\\

\noindent
The key dynamics what we have achieved in the complex plane are summarized here \dots

\begin{itemize}
  \item There are three fixed points of the dynamical systems (\ref{equation:total-equationB}) and there are certain parameters $\alpha, \beta, \gamma$ and $\delta$ (examples are given) such that those fixed points are stable (sink) and unstable (source).
  \item The dynamical system (\ref{equation:total-equationB}) possess higher order periodic solutions too (few examples are given).
  \item The dynamical system (\ref{equation:total-equationB}) has chaotic (fractal like and fractal-unlike) solutions.
\end{itemize}

\noindent
In the following sections, a detail dynamics of the dynamical system (\ref{equation:total-equationB}) is characterized.

\section{Local Asymptotic Stability of the Fixed Points}

The fixed points of the family of maps (\ref{equation:total-equationA}) are the solutions of the cubic equation
\[
\bar{z}=\frac{\alpha \bar{z}+\beta}{\gamma \bar{z}^2+\delta \bar{z}}
\]
The map (\ref{equation:total-equationA}) has the three fixed points $\bar{z}_{1,2,3}$ \dots \\
\small

\noindent
{$ -\frac{\delta }{3 \gamma }-\frac{2^{1/3} \left(-3 \alpha  \gamma -\delta ^2\right)}{3 \gamma  \left(27 \beta  \gamma ^2-9 \alpha  \gamma  \delta -2 \delta ^3+\sqrt{4 \left(-3 \alpha  \gamma -\delta ^2\right)^3+\left(27 \beta  \gamma ^2-9 \alpha  \gamma  \delta -2 \delta ^3\right)^2}\right)^{1/3}}+\frac{\left(27 \beta  \gamma ^2-9 \alpha  \gamma  \delta -2 \delta ^3+\sqrt{4 \left(-3 \alpha  \gamma -\delta ^2\right)^3+\left(27 \beta  \gamma ^2-9 \alpha  \gamma  \delta -2 \delta ^3\right)^2}\right)^{1/3}}{3\ 2^{1/3} \gamma }$, \\ $-\frac{\delta }{3 \gamma }+\frac{\left(1+i \sqrt{3}\right) \left(-3 \alpha  \gamma -\delta ^2\right)}{3\ 2^{2/3} \gamma  \left(27 \beta  \gamma ^2-9 \alpha  \gamma  \delta -2 \delta ^3+\sqrt{4 \left(-3 \alpha  \gamma -\delta ^2\right)^3+\left(27 \beta  \gamma ^2-9 \alpha  \gamma  \delta -2 \delta ^3\right)^2}\right)^{1/3}}-\frac{\left(1-i \sqrt{3}\right) \left(27 \beta  \gamma ^2-9 \alpha  \gamma  \delta -2 \delta ^3+\sqrt{4 \left(-3 \alpha  \gamma -\delta ^2\right)^3+\left(27 \beta  \gamma ^2-9 \alpha  \gamma  \delta -2 \delta ^3\right)^2}\right)^{1/3}}{6\ 2^{1/3} \gamma }$, $-\frac{\delta }{3 \gamma }+\frac{\left(1-i \sqrt{3}\right) \left(-3 \alpha  \gamma -\delta ^2\right)}{3\ 2^{2/3} \gamma  \left(27 \beta  \gamma ^2-9 \alpha  \gamma  \delta -2 \delta ^3+\sqrt{4 \left(-3 \alpha  \gamma -\delta ^2\right)^3+\left(27 \beta  \gamma ^2-9 \alpha  \gamma  \delta -2 \delta ^3\right)^2}\right)^{1/3}}-\frac{\left(1+i \sqrt{3}\right) \left(27 \beta  \gamma ^2-9 \alpha  \gamma  \delta -2 \delta ^3+\sqrt{4 \left(-3 \alpha  \gamma -\delta ^2\right)^3+\left(27 \beta  \gamma ^2-9 \alpha  \gamma  \delta -2 \delta ^3\right)^2}\right)^{1/3}}{6\ 2^{1/3} \gamma }$ respectively.\\

\noindent
The linearized equation of dynamical system Eq.(\ref{equation:total-equationB}) with respect to the a fixed point $\bar{z_{i}}$ for $i=1,2$ and $3$ is

\begin{equation}
\label{equation:linearized-equation}
\displaystyle{
z_{n+1}=f'_{\alpha, \beta, \gamma, \delta}(z_{n}), n=0,1,\ldots
}
\end{equation}

\begin{theorem}
The fixed points $\bar{z_{i}}$ for $i=1,2$ and $3$ of the Eq.(\ref{equation:total-equationB}) is \dots \\
\emph{locally asymptotically stable} if $$\abs{f'_{\alpha, \beta, \gamma, \delta}(\bar{z_{i}})}<1$$
\emph{unstable} if $$\abs{f'_{\alpha, \beta, \gamma, \delta}(\bar{z_{i}})}>1$$
\emph{non-hyperbolic} if $$\abs{f'_{\alpha, \beta, \gamma, \delta}(\bar{z_{i}})}=1$$
\end{theorem}

\begin{remark}
It is observed that the minimum value of the $\abs{f'_{\alpha, \beta, \gamma, \delta}(\bar{z_{1}})}$ is $6.57147 \times 10^{-9}$ which is less than $1$ when the parameters are taken as $\alpha=0.816885-0.0738146i, \beta=-0.0245487+0.0678909i,\gamma=-0.360452-0.788031i$ and $\delta=1.89324+0.104191i$.\\
\noindent
This numerical observation makes a guarantee that there is parameters $\alpha$, $\beta$, $\gamma$ and $\delta$ such that $\abs{f'_{\alpha, \beta, \gamma, \delta}(\bar{z_{1}})}<1$ holds good. Therefore existence of parameters is ensured for local asymptotic stability (sink) of the fixed point $\bar{z_{1}}$ of the dynamical system Eq.(\ref{equation:total-equationB}).\\
\end{remark}

\begin{example}
Let the parameters be $\alpha=0.92735+0.9174938i,\beta=0.713574+0.618337i,\gamma=0.343287+0.9360273i$ and $\delta=0.124774+0.7305853i$, then the fixed points of the dynamical system Eq.(\ref{equation:total-equationB}) are $\{\bar{z}_1=-0.708428+0.171918 i\},\{ \bar{z}_2=1.1079 -0.305049 i\}$ and $\{\bar{z}_3 =-1.13054-0.00168788 i\}$. \\ It is noted that the $\abs{\alpha}=1.3045$, $\abs{\beta}=0.9442$, $\abs{\gamma}=0.9970$, $\abs{\delta}=0.7412$, $\abs{\alpha+\beta}=2.2475$ and $\abs{\gamma+\delta}=1.7311$ and hence $\abs{\alpha}>\abs{\gamma}$, $\abs{\beta}>\abs{\delta}$, $\abs{\alpha+\beta}>\abs{\gamma+\delta}$. Here $\abs{f'_{\alpha, \beta, \gamma, \delta}(\bar{z_{1}})}$ is $3.51012$ which is greater than $1$. Hence the fixed point is \emph{unstable (source)}.\\

\end{example}

\noindent
The trajectory plots of $10000$ and $40000$ iterations with an initial value have been figures in Fig. $1$. \\

\begin{figure}[H]
      \centering

      \resizebox{12cm}{!}
      {
      \begin{tabular}{c c}
      \includegraphics [scale=6]{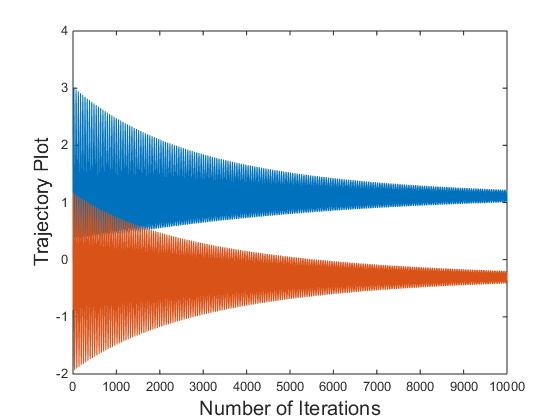}
      \includegraphics [scale=6]{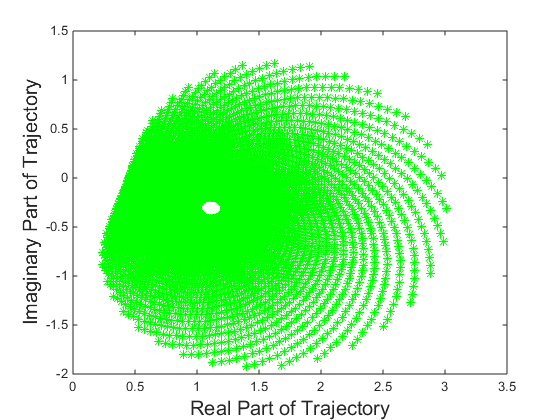}\\
      \includegraphics [scale=6]{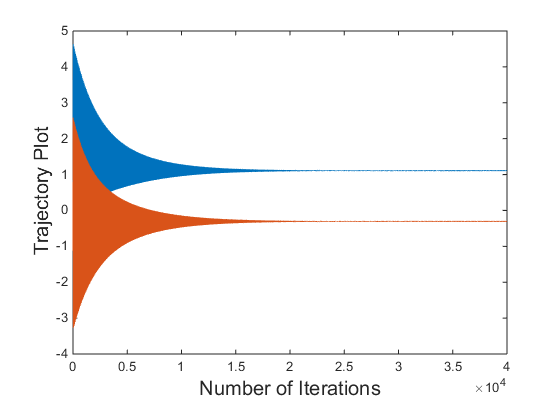}
      \includegraphics [scale=6]{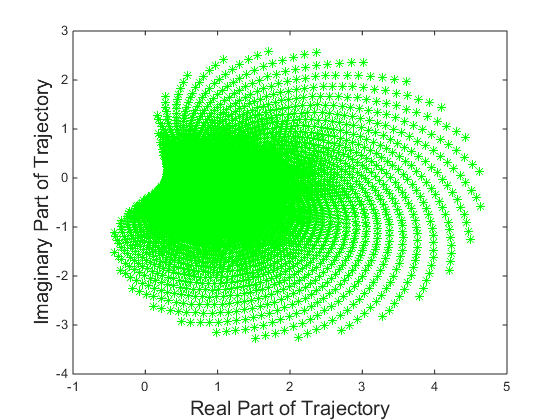}\\
            \end{tabular}
      }
\caption{Trajectory Plots.}
      \begin{center}

      \end{center}
      \end{figure}

\noindent
It is seen in the Fig. 1 that the trajectory plots are very unstable in nature while it is converging to the fixed point $-0.708428+0.171918i$.

\begin{remark}
It is observed that the maximum value of the $\abs{f'_{\alpha, \beta, \gamma, \delta}(\bar{z_{2}})}$ is $2.70948 \times 10^{11}$ which is much greater than $1$ when the parameters are taken as $\alpha=-0.0000775242-0.0000567123i, \beta=-1.66869+0.581375i,\gamma=1.33603 \times 10^{-6}-6.20472 \times 10^{-7}i$ and $\delta=-5.94147-1.73645i$. \\
\noindent
This numerical observation confirms that there is parameters $\alpha$, $\beta$, $\gamma$ and $\delta$ such that $\abs{f'_{\alpha, \beta, \gamma, \delta}(\bar{z_{2}})}>1$ holds good. Therefore existence of parameters is ensured for unstable (source) solution at the fixed point $\bar{z_{2}}$ of the dynamical system Eq.(\ref{equation:total-equationB}).\\\\
\end{remark}

\begin{example}
Consider $\alpha \to 0.27481+0.24150174i,\beta \to 0.243145+0.154159i,\gamma \to 0.956416+0.935661i,\delta \to 0.818714+0.728261i$ then the fixed points are $\{\bar{z_{1}} \to 0.515402 -0.0307232 i\},\{ \bar{z_{2}}\to -0.484732+0.0782783 i\},\{\bar{z_{3}} \to -0.848703-0.00872243 i\}$.\\ It is noted that the $\abs{\alpha}=0.3658$, $\abs{\beta}=0.2879$, $\abs{\gamma}=1.3380$, $\abs{\delta}=1.0957$, $\abs{\alpha+\beta}=0.6518$ and $\abs{\gamma+\delta}=2.4330$ and hence $\abs{\alpha}<\abs{\gamma}$, $\abs{\beta}<\abs{\delta}$, $\abs{\alpha+\beta}<\abs{\gamma+\delta}$. Here $\abs{f'_{\alpha, \beta, \gamma, \delta}(\bar{z_{1}})}$ is $0.991066$ which is lesser than $1$. Hence the fixed point is \emph{locally asymptotically stable (sink)}.
\end{example}

\begin{figure}[H]
      \centering

      \resizebox{15cm}{!}
      {
      \begin{tabular}{c c}
      \includegraphics [scale=6]{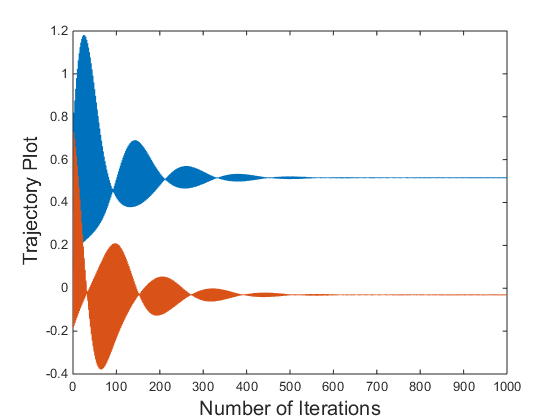}
      \includegraphics [scale=6]{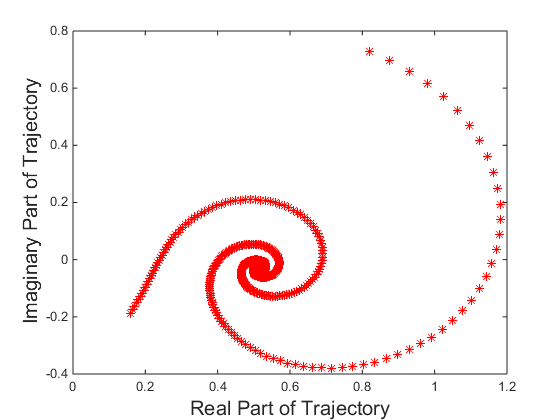}\\

            \end{tabular}
      }
\caption{Trajectory Plots.}
      \begin{center}

      \end{center}
      \end{figure}

\noindent
In Fig.$2$, the trajectory plot for an initial value of $1000$ iterations is given. it is seen in the Fig.$2$ that the trajectory plot is very stable in nature while it is converging to the fixed point $0.515402 -0.0307232i$.\\

\noindent

\section{Verification of Different Criteria}
In this section, we shall explore the convergence, divergence and constant solution conditions for the dynamical system Eq.(\ref{equation:total-equationB}) which are mentioned in the section $1$. \\

\subsection{Verification of Convergence Criteria}
Consider the parameters of the dynamical system Eq.(\ref{equation:total-equationB}) $\alpha=0.27481 + 0.241501i$ ($\abs{\alpha}=0.3658$), $\beta=1.2431 + 0.1542i$ ($\abs{\beta}=1.2526$), $\gamma=0.956416 + 0.935661i$ ($\abs{\gamma}=1.3380$) and $\delta=0.818714 + 0.728261i$ ($\abs{\delta}=1.0957$), $\abs{\alpha+\beta}=1.5686$ and $\abs{\gamma+\delta}=2.4330$. \\ Here $\abs{\alpha}<\abs{\gamma}$, $\abs{\beta}<\abs{\delta}$, $\abs{\alpha+\beta}<\abs{\gamma+\delta}$ which are the necessary conditions to have convergent solution which converge to zero as stated in the section $1$. But as it can be seen that the solution in this case in the complex plane is unbounded. The trajectory plots for $20$ different initial values are given in Fig.$3$.

\begin{figure}[H]
      \centering

      \resizebox{6cm}{!}
      {
      \begin{tabular}{c}
      \includegraphics [scale=3]{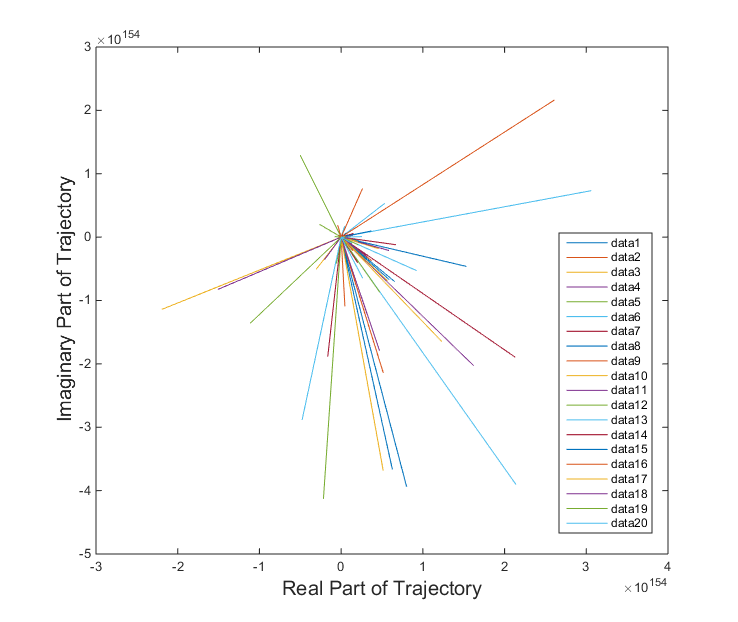}
            \end{tabular}
      }
\caption{Trajectory Plot of Unbounded Solution.}
      \begin{center}

      \end{center}
      \end{figure}

\noindent
Here none of the twenty trajectories is convergent and converges to zero, all are unbounded indeed. This does not nullify the result which was in the real line but it encounters trajectories on contrary in the complex plane.

\subsection{Verification of Divergence Criteria}
Consider the parameters of the dynamical system Eq.(\ref{equation:total-equationB}) $\alpha=0.917193 + 0.2858390i$ ($\abs{\alpha}=0.9607$), $\beta=1.13764 + 1.32155i$ ($\abs{\beta}=1.7438$), $\gamma=0.993047 + 0.33978i$ ($\abs{\gamma}=1.0496$) and $\delta=0.7572 + 0.753729i$ ($\abs{\delta}=1.0684$), $\abs{\alpha+\beta}=2.6088$ and $\abs{\gamma+\delta}=2.063$. \\ Here it is found that $\abs{\alpha}<\abs{\gamma}$, $\abs{\beta}<\abs{\delta}$, $\abs{\alpha+\beta}>\abs{\gamma+\delta}$ which are the necessary conditions to have a divergent solution and which diverges to infinity as stated in the section $1$. Here as it can be seen that the solution in the complex plane is converging to infinity. The trajectory plots for $50$ different initial values are given in Fig.$4$.

\begin{figure}[H]
      \centering

      \resizebox{6cm}{!}
      {
      \begin{tabular}{c}
      \includegraphics [scale=3]{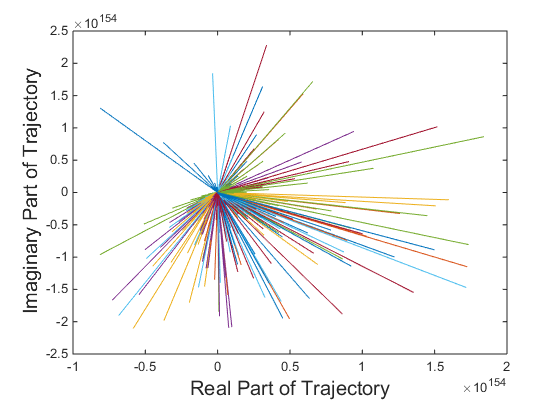}
            \end{tabular}
      }
\caption{Trajectory Plot of Unbounded Solution.}
      \begin{center}

      \end{center}
      \end{figure}

\noindent
Here all the fifty trajectories is divergent and diverges to infinity. This does justify the truth of result in the complex plane.

\subsection{Verification of Constant Criteria}
Consider the parameters of the dynamical system Eq.(\ref{equation:total-equationB}) $\alpha=0.9172 + 0.2858i$ ($\abs{\alpha}=0.9607$), $\beta=1.1376 + 1.3216i$ ($\abs{\beta}=1.7438$), $\gamma=3.1376 + 1.3216i$ ($\abs{\gamma}=3.4046$) and $\delta=-1.0828 + 0.2858i$ ($\abs{\delta}=1.1199$), $\abs{\alpha+\beta}=\abs{\gamma+\delta}=2.6088$. \\ Here $\abs{\alpha}<\abs{\gamma}$, $\abs{\beta}<\abs{\delta}$, $\abs{\alpha+\beta}=\abs{\gamma+\delta}$ which are the necessary conditions to have a constant solution as stated in the section $1$. \\ But here it is seen that the solution in this case in the complex plane is divergent and diverges to infinity. The trajectory plots for $50$ different initial values are given in Fig.$5$. \\

\begin{figure}[H]
      \centering

      \resizebox{6cm}{!}
      {
      \begin{tabular}{c}
      \includegraphics [scale=3]{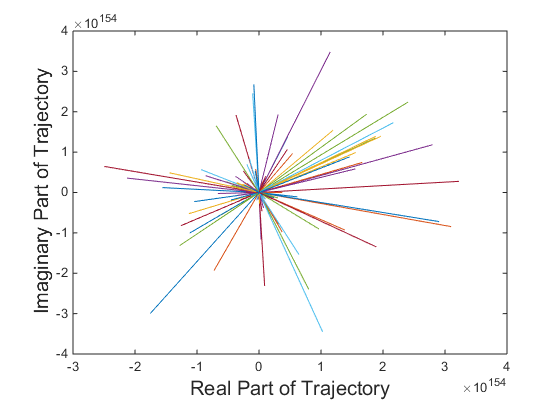}
            \end{tabular}
      }
\caption{Trajectory Plot of Unbounded Solution.}
      \begin{center}

      \end{center}
      \end{figure}

\noindent
Here none of the fifty trajectories is constant. This does not nullify the result in the complex plane which was in real case but it encounters trajectories on its contrary.

\section{Periodic of Solutions}

A solution $\{z_n\}_n$ of a dynamical system is said to be \emph{globally periodic} of period $t$ if $z_{n+t}=z_n$ for any given initial conditions. A solution $\{z_n\}_n$ is said to be \emph{periodic with prime period} $p$ if p is the smallest positive integer having this property.\\

\noindent
As mentioned already that in the real line, there is no mention of periodic solutions of the dynamical system Eq.(\ref{equation:total-equationB}). Here a computational attempt has been made in the complex plane scenario to search for the periodic solutions if it has any. \\

\begin{theorem}
There does not exists any period two solutions (non-trivial) of the dynamical system Eq.(\ref{equation:total-equationB}).
\end{theorem}

\begin{proof}
Suppose $\phi$ an $\psi$ are two period 2-cycle solution of the dynamical system Eq.(\ref{equation:total-equationB}). Therefore by definition of the period point, $\phi$ and $\psi$ satisfy two equations $$\phi=\frac{\beta +\alpha  \psi}{\delta  \psi+\gamma  \phi^2}, \psi=\frac{\beta +\alpha  \phi}{\delta  \phi+\gamma  \psi^2}$$
It i trivial to observe that except the fixed points of the map $f_{\alpha, \beta, \gamma, \delta}(z)=\frac{\alpha z + \beta}{\gamma z^2 +\delta z}$, there is no other solution of these system of two linear equations. \\ Hence there does not exists any non-trivial period $2$-cycle solution of the Eq.(\ref{equation:total-equationB}).
\end{proof}

\subsection{Example Cases of Periodic Solutions of Different Periods}
\textbf{Example 4.1:}
Consider the parameters $\alpha=0.655098003973841 + 0.162611735194631i$ ($\abs{\alpha}=0.674978$), $\beta=0.118997681558377 + 0.498364051982143i$ ($\abs{\beta}=0.512374$), $\gamma=\beta$ and $\delta=\alpha$, $\abs{\alpha+\beta}=\abs{\gamma+\delta}=1.0179$. Here $\abs{\alpha}>\abs{\gamma}$, $\abs{\beta}<\abs{\delta}$, $\abs{\alpha+\beta}=\abs{\gamma+\delta}$. \\ This set up of parameters produce a periodic solution of period $5$. One of the periodic 5-cycles is $\{0.14044571-0.62692799i, 0.02078534+0.517795i, 2.994135-1.828805i, 0.27873451-0.0515247i, 1.8686505+1.9830857i$ and $0.14044571-0.62692799i\}$. The trajectory plots for five different initial values are given in the Fig.$6$.

\begin{figure}[H]
      \centering

      \resizebox{12cm}{!}
      {
      \begin{tabular}{c c}
      \includegraphics [scale=5]{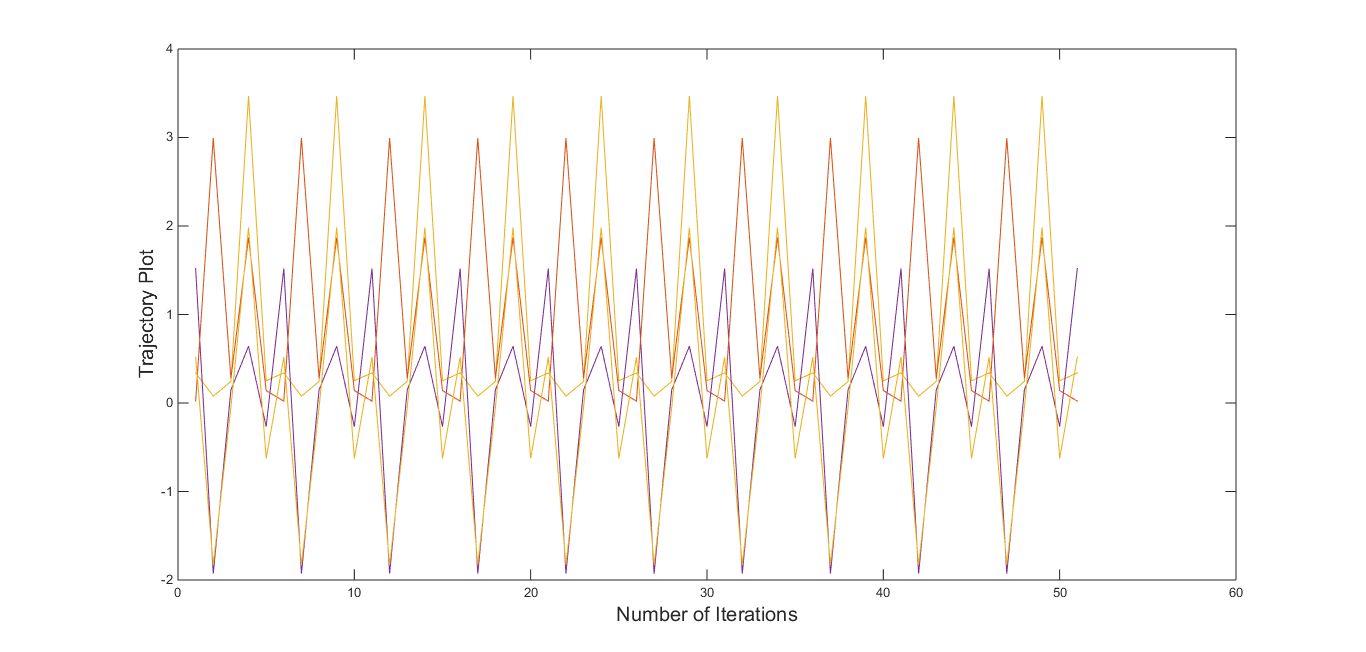}
      \includegraphics [scale=6]{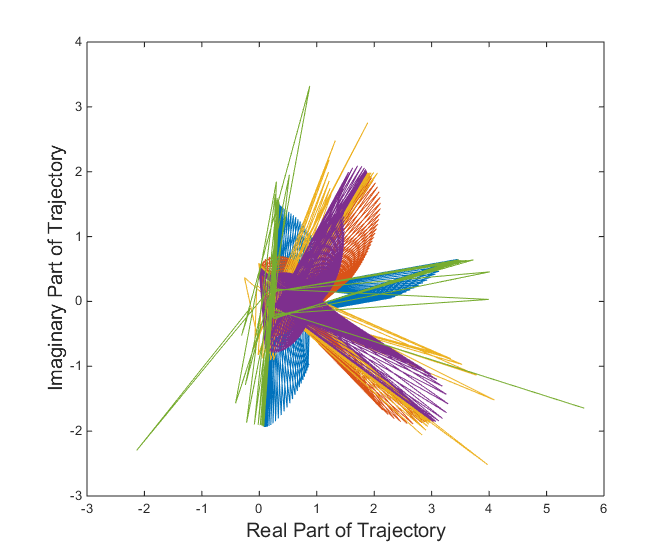}\\
      \end{tabular}
      }
\caption{Periodic trajectory plot (from 800 to 850 iterations of total 1000 iterations) of period 5.}
      \begin{center}

      \end{center}
      \end{figure}

\noindent
\textbf{Example 4.2:} Consider the parameters $\alpha=0.1909 + 0.4283i$ ($\abs{\alpha}=0.4689$), $\beta=0.4820 + 0.1206i
$ ($\abs{\beta}=0.4969$), $\gamma=0.5895 + 0.2262i$ ($\abs{\gamma}=0.6314$) and $\delta=0.3846 + 0.5830i$, ($\abs{\delta}=0.6984$) $\abs{\alpha+\beta}=0.8684$ $\abs{\gamma+\delta}=1.2664$. Here $\abs{\alpha}<\abs{\gamma}$, $\abs{\beta}<\abs{\delta}$, $\abs{\alpha+\beta}<\abs{\gamma+\delta}$. \\  This set up of parameters produce a periodic solution of period $8$. One of the periodic 8-cycles is $\{0.1245073+0.0238456i, 3.7977682-3.81535i, 0.0328615+0.1531786i, -1.76641-3.292791i, -0.296000+0.07258225i, -1.224311+1.8140937i, 0.08601284-0.129181i, 4.8712818+1.8954348i$ and $0.1245073+0.0238456i\}$. The trajectory plots for five different initial values are given in the Fig.$7$.

\begin{figure}[H]
      \centering

      \resizebox{12cm}{!}
      {
      \begin{tabular}{c c}
      \includegraphics [scale=5]{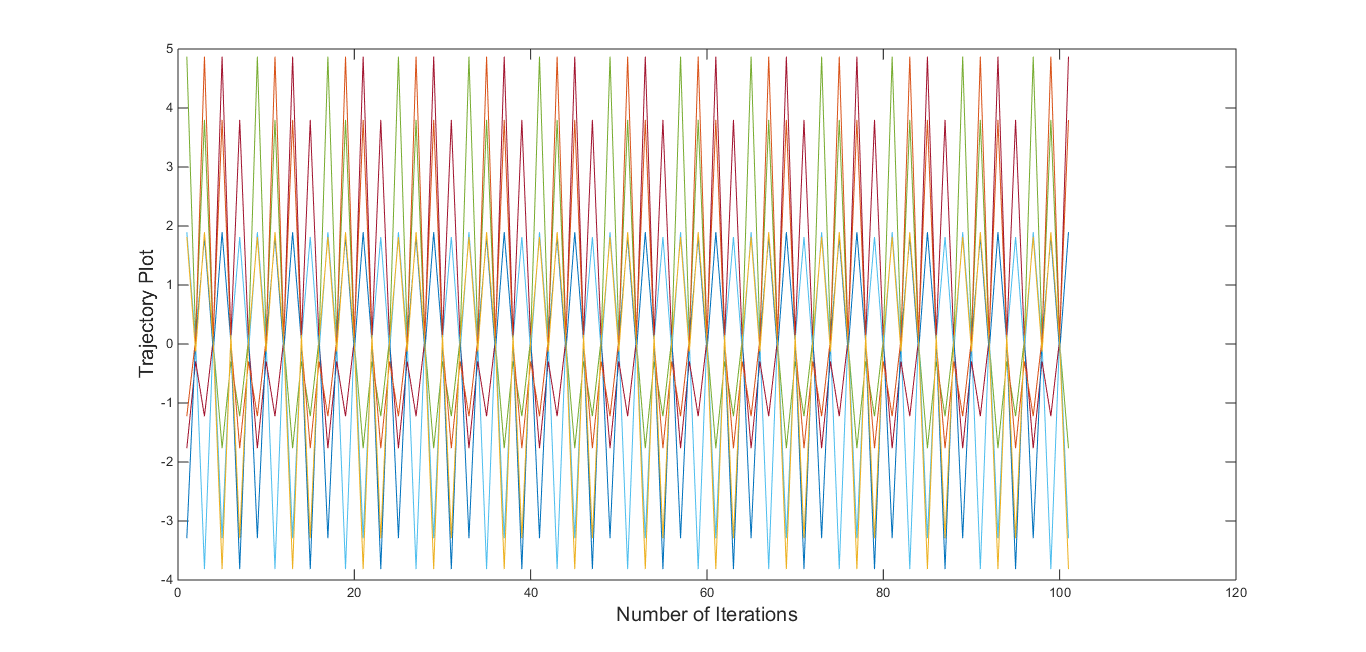}
      \includegraphics [scale=8]{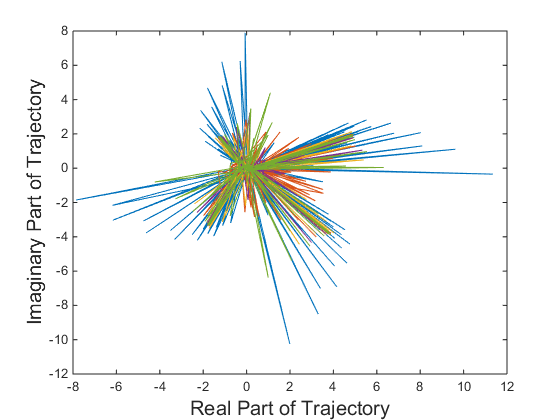}\\
      \end{tabular}
      }
\caption{Periodic trajectory plot (from 900 to 1000 iterations of total 1000 iterations) of period 8.}
      \begin{center}

      \end{center}
      \end{figure}

\noindent
\textbf{Example 4.3:} Consider the parameters $\alpha=0.1557 + 0.8190i$ ($\abs{\alpha}=0.8337$), $\beta=0.6249 + 0.7386i$ ($\abs{\beta}= 0.9675$), $\gamma=0.8051 + 0.0672i$, ($\abs{\gamma}=0.8079$) and $\delta=0.9508 + 0.4976i$, ($\abs{\delta}=1.0731$) $\abs{\alpha+\beta}=1.74224$ $\abs{\gamma+\delta}=1.8445$. \\ Here $\abs{\alpha}>\abs{\gamma}$, $\abs{\beta}<\abs{\delta}$, $\abs{\alpha+\beta}<\abs{\gamma+\delta}$. This set up of parameters produce a periodic solution of period $8$. One of the periodic 13-cycles is $\{ 0.4995-1.9204i, -0.4405+0.5198i, -0.2696-0.6103i, -1.6184+1.5608i, 0.1910-0.1647i, 1.3640+3.7100i, 0.2073-0.0035i, 3.7982+2.2357i, 0.1717+0.1174i, 4.0798-0.3778i, 0.0933+0.2274i, 2.4590-2.0608i, -0.0609+0.3453i, 0.4995-1.9204i \}$. The trajectory plots for five different initial values are given in the Fig.$8$.

\begin{figure}[H]
      \centering

      \resizebox{12cm}{!}
      {
      \begin{tabular}{c c}
      \includegraphics [scale=5]{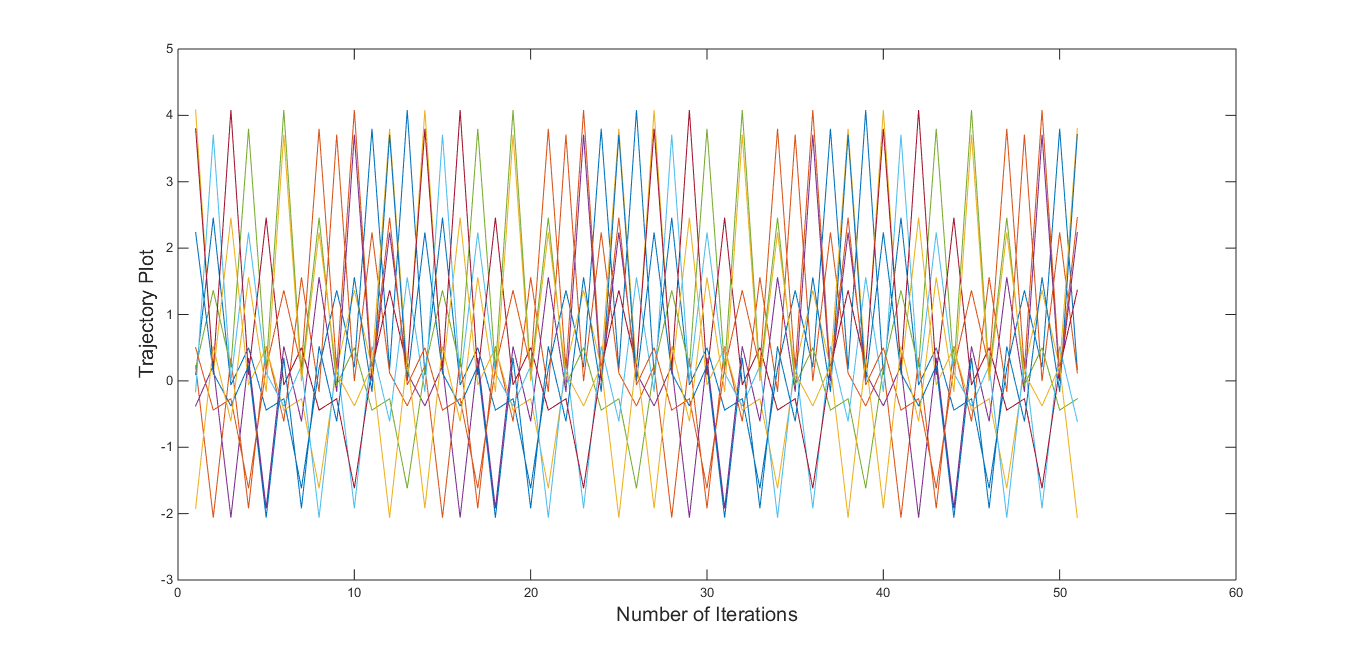}
      \includegraphics [scale=8]{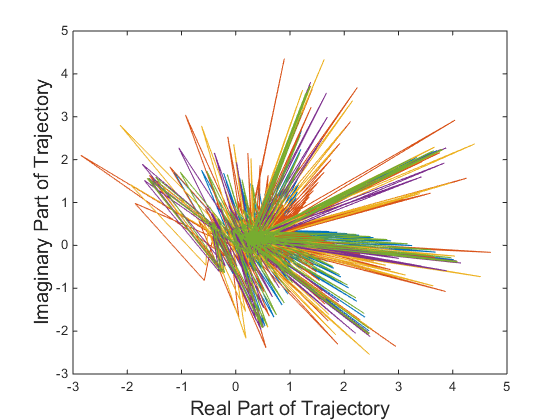}\\
      \end{tabular}
      }
\caption{Periodic trajectory plot (from 1450 to 1500 iterations of total 1500 iterations) of period 13.}
      \begin{center}

      \end{center}
      \end{figure}

\noindent
\textbf{Example 4.4:} Consider the parameters $\alpha=0.6228 + 0.7966i$ ($\abs{\alpha}=1.0112$), $\beta=0.7459 + 0.1255i$ ($\abs{\beta}=0.7564$), $\gamma=0.8224 + 0.0252i$ ($\abs{\gamma}=0.8228$) and $\delta=0.4144 + 0.7314i$, ($\abs{\delta}= 0.8407$) $\abs{\alpha+\beta}=1.6504$ $\abs{\gamma+\delta}=1.4499$. \\ Here $\abs{\alpha}>\abs{\gamma}$, $\abs{\beta}<\abs{\delta}$, $\abs{\alpha+\beta}>\abs{\gamma+\delta}$. This set up of parameters produce a periodic solution of period $8$. One of the periodic $21$-cycles is $\{ 0.1508-0.2674i, 4.2660+1.6796i, 0.2370+0.0609i, 2.6997-2.6189i, 0.0772+0.3937i, -0.3141-1.3571i, -2.2634+2.1519i, 0.0920-0.2518i, 4.6679+2.2400i, 0.2107+0.0451i, 3.0966-2.9892i, 0.0554+0.3390i, -0.5706-1.5730i, -2.0809+0.6391i, 0.1141-0.3835i, 3.5201+2.3141i, 0.2519+0.0117i, 3.2825-2.0879i, 0.1469+0.3346i, -0.0137-1.7473i, -0.8858+1.3368i$ and $0.1508-0.2674i\}$. The trajectory plot for five different initial values is given in the Fig.$9$.

\begin{figure}[H]
      \centering

      \resizebox{12cm}{!}
      {
      \begin{tabular}{c c}
      \includegraphics [scale=5]{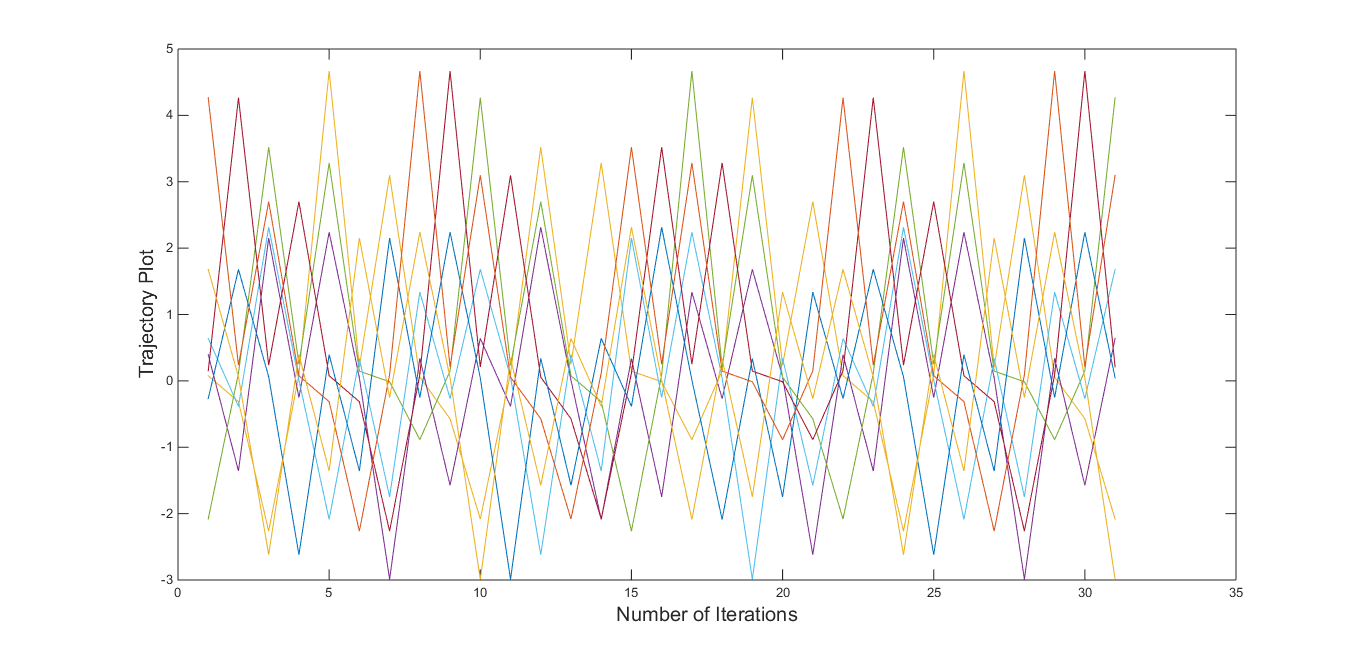}
      \includegraphics [scale=8]{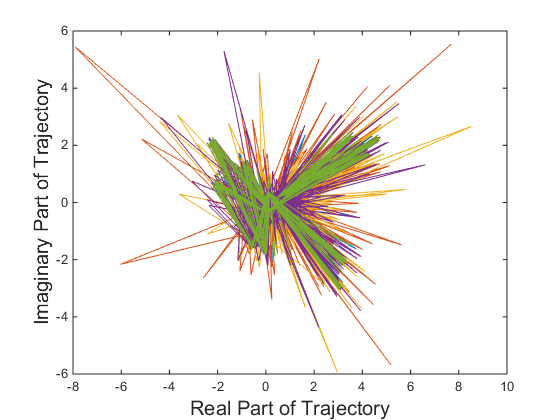}\\
      \end{tabular}
      }
\caption{Periodic trajectory plot (from 1970 to 2000 iterations of total 2000 iterations) of period 21.}
      \begin{center}

      \end{center}
      \end{figure}

\noindent
In the above four different examples, we found different periodic trajectories/solutions of the dynamical system Eq.(\ref{equation:total-equationB}). \\
In the \emph{Example 4.1}, we saw the parameters are obeying $\abs{\alpha}>\abs{\gamma}$, $\abs{\beta}<\abs{\delta}$ and $\abs{\alpha+\beta}=\abs{\gamma+\delta}$ conditions which did not arise in real scenario.\\
In \emph{Example 4.2}, the parameters are satisfying $\abs{\alpha}<\abs{\gamma}$, $\abs{\beta}<\abs{\delta}$ and $\abs{\alpha+\beta}<\abs{\gamma+\delta}$ which condition did not appear in the real set up. \\
In \emph{Example 4.3}, the parameters are satisfying $\abs{\alpha}>\abs{\gamma}$, $\abs{\beta}<\abs{\delta}$ and $\abs{\alpha+\beta}<\abs{\gamma+\delta}$ which condition did not appear in the real set up.\\
In \emph{Example 4.4}, the parameters are satisfying $\abs{\alpha}>\abs{\gamma}$, $\abs{\beta}<\abs{\delta}$ and $\abs{\alpha+\beta}>\abs{\gamma+\delta}$ which condition did not appear in the real set up.\\

\noindent
It is to be noted that in all the above four examples, the parameters satisfy the $\abs{\alpha}>\abs{\gamma}$ or $\abs{\beta}<\abs{\delta}$ conditions which did not arise in real scenario \cite{D-C-N}.

\section{Slow and Fast Convergence}
Here we encounter two solutions of the Eq.(\ref{equation:total-equationB}), one of which is very slow in converging to the fixed point and another one is very fast inc converging to the fixed point of the dynamical system.\\
Here the parameters $\alpha=0.9274 + 0.9175i, (\abs{\alpha}=1.3045)$, $\beta=0.7136 + 0.6183i, (\abs{\beta}=0.9442)$, $\gamma=0.3433 + 0.9360i, (\abs{\gamma}=0.9970)$, $\delta=0.1248 + 0.7306i, (\abs{\delta}=0.7412)$ and $\abs{\alpha+\beta}=2.247$ and $\abs{\gamma+\delta}=1.7310$. That is the parameters are satisfying $\abs{\alpha}<\abs{\gamma}$, $\abs{\beta}>\abs{\delta}$ and ${\abs{\alpha+\beta}>\abs{\gamma+\delta}}$ which was the necessary creation for convergence to \emph{zero} in the real set up but here under this set of parameters the solution are convergent for any initial values to $1.10789-0.305048i$ (\emph{not zero}!) with about $10^{5}$ number of iterations which is slow indeed.\\
The trajectory plots with ten different arbitrary initial values are given in the following figure Fig.$10$.\\

\begin{figure}[H]
      \centering

      \resizebox{12cm}{!}
      {
      \begin{tabular}{c c}
      \includegraphics [scale=5]{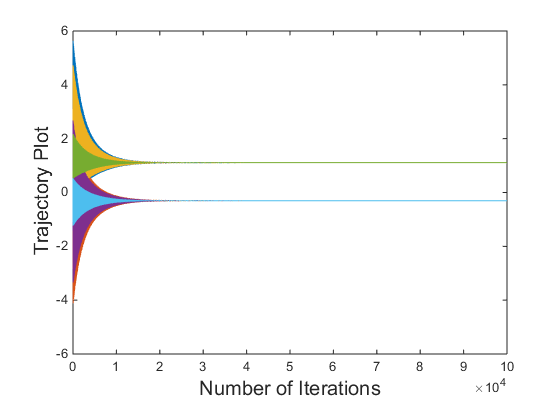}
      \includegraphics [scale=5]{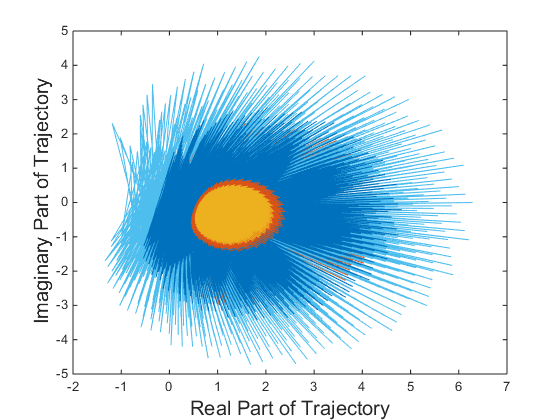}\\
      \end{tabular}
      }
\caption{Slow convergent trajectory plot.}
      \begin{center}

      \end{center}
      \end{figure}

\noindent
Here it is noted that the trajectory is unstable for any initial values since it the condition is $\abs{f'_{\alpha, \beta, \gamma, \delta}(\bar{z_{i}})}=7.45235>1$ satisfied.

\noindent
Here the parameters $\alpha=0.2748 + 0.2415i, (\abs{\alpha}=0.3658)$, $\beta=0.2431 + 0.1542i, (\abs{\beta}=0.2879)$, $\gamma=0.9564 + 0.9357i, (\abs{\gamma}=1.3380)$, $\delta=0.8187 + 0.7283i, (\abs{\delta}=1.0957)$ and $\abs{\alpha+\beta}=0.6517$ and $\abs{\gamma+\delta}=2.433$. That is the parameters are satisfying $\abs{\alpha}>\abs{\gamma}$, $\abs{\beta}<\abs{\delta}$ and ${\abs{\alpha+\beta}<\abs{\gamma+\delta}}$ which condition did not arise at all in real scenario but here under this set of parameters, the solutions are convergent for any initial values to $0.515394-0.03072$ with about $10^{3}$ number of iterations which is fast indeed. The trajectory plots with ten different arbitrary initial values are given in the following figure Fig.$11$.\\\\\\

\begin{figure}[H]
      \centering

      \resizebox{12cm}{!}
      {
      \begin{tabular}{c c}
      \includegraphics [scale=5]{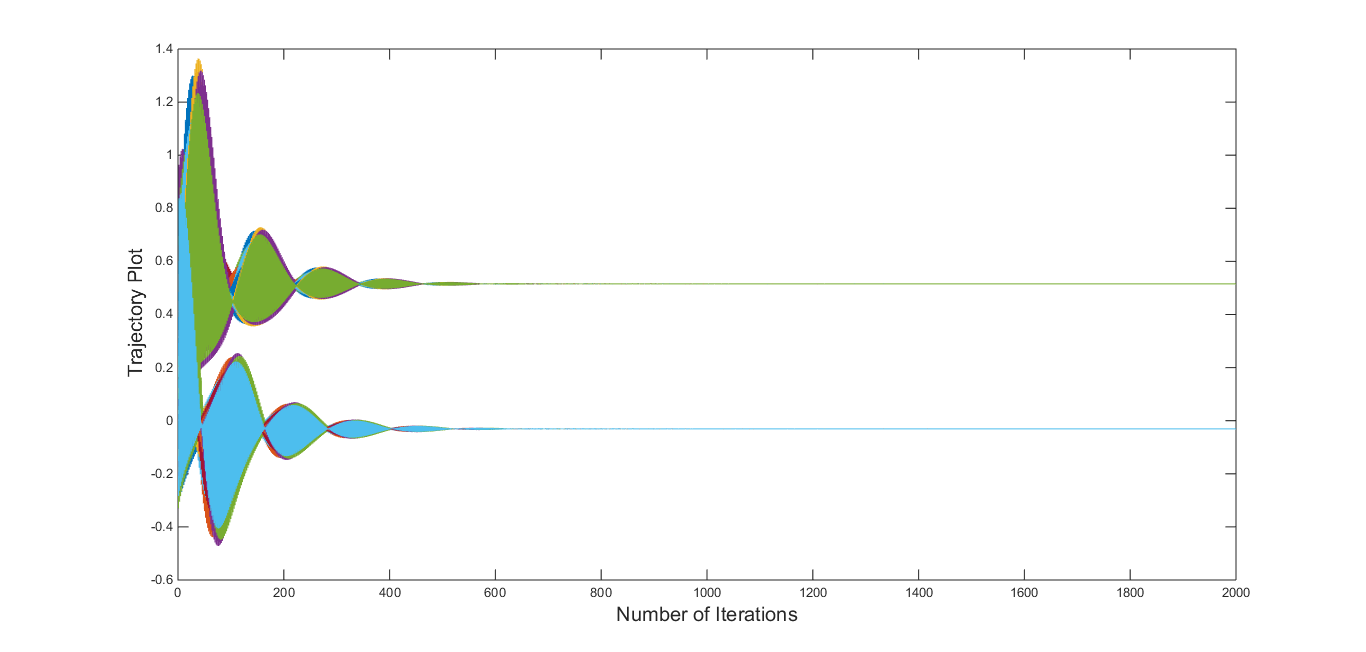}
      \includegraphics [scale=8]{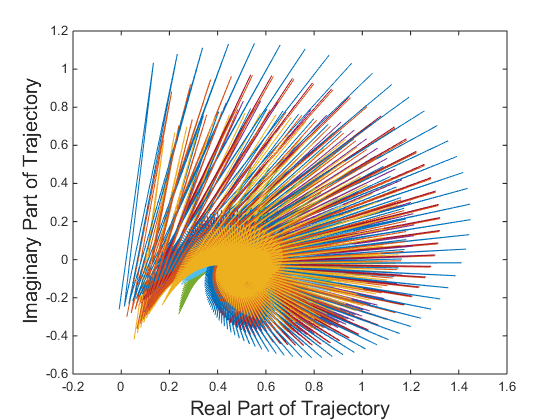}\\
      \end{tabular}
      }
\caption{Fast convergent trajectory plot.}
      \begin{center}

      \end{center}
      \end{figure}
\noindent
Here it is noted that the all the ten trajectories are unstable since it the condition is $\abs{f'_{\alpha, \beta, \gamma, \delta}(\bar{z_{i}})}=4.85602>1$ satisfied.

\noindent
It is interesting to note that in both cases above, the trajectories are unstable as $\abs{f'_{\alpha, \beta, \gamma, \delta}(\bar{z_{i}})}>1$ but in one case the convergence is slow and in another case is fast.

\section{Chaotic Solutions}
Here we have explored the existence of chaotic solutions of the dynamical system Eq.(\ref{equation:total-equationB}). Computationally, some chaotic solutions of the dynamical system Eq.(\ref{equation:total-equationB}) for some parameters $\alpha$, $\beta$, $\gamma$ and $\delta$ which are given in the following Table $1$ are fetched. \\
\noindent
The largest Lyapunov exponent is calculated for all such solutions of the dynamical system Eq.(\ref{equation:total-equationA}) numerically \cite{Wolf} to show the trajectories are chaotic.\\
\noindent
From computational evidence, it is arguable that for complex parameters $\alpha$, $\beta$, $\gamma$ and $\delta$ which are stated in the following Table $1$, the solutions are chaotic for any initial values.\\

\begin{table}[H]

\begin{tabular}{| m{8cm}||m{2.2cm}| |m{4.5cm}|}
\hline
\centering   \textbf{Parameters} &
\begin{center}
\textbf{Parameters Conditions}
\end{center}
 &
\begin{center}
\textbf{Lyapunav exponent (Fractal like/unlike}
\end{center} \\
\hline
\centering $\alpha=0.6849 + 0.2083i$, $\beta=0.6082 + 0.3262i$, $\gamma=0.8808 + 0.1334i
$, $\delta=0.1024 + 0.9591i$, $\abs{\alpha}=0.7158$, $\abs{\beta}=0.6901$, $\abs{\gamma}=0.8909$, $\abs{\delta}=0.9646$, $\abs{\alpha+\beta}=1.3991$ \& $\abs{\gamma+\delta}=1.4698$ & \begin{center}
$\abs{\alpha}<\abs{\gamma}$, $\abs{\beta}<\abs{\delta}$, $\abs{\alpha+\beta}<\abs{\gamma+\delta}$
\end{center}
&
\begin{center}
Lyapunav Exponent:$1.3342$ \\ (Fig.12: Fractal-unlike)
\end{center}\\
\hline
\centering $\alpha=0.8491 + 0.9340i$, $\beta=0.6787 + 0.7577i$, $\gamma=0.7431 + 0.3922i
$, $\delta=0.6555 + 0.1712i$, $\abs{\alpha}=1.2623$, $\abs{\beta}=1.0173$, $\abs{\gamma}=0.8403$, $\abs{\delta}=0.6775$, $\abs{\alpha+\beta}=2.2795$ \& $\abs{\gamma+\delta}=1.5078$ & \begin{center}
$\abs{\alpha}>\abs{\gamma}$, $\abs{\beta}>\abs{\delta}$, $\abs{\alpha+\beta}>\abs{\gamma+\delta}$
\end{center}
&
\begin{center}
Lyapunav Exponent: $1.4765$, Fractal dimension:$1.3435$\\ (Fig.13: Fractal-like)
\end{center}\\
\hline
\centering $\alpha=0.9322 + 0.8351i$, $\beta=0.8954 + 0.5825i$, $\gamma=0.5827 + 0.8549i
$, $\delta=0.0349 + 0.8854i$, $\abs{\alpha}=1.2515$, $\abs{\beta}=1.0682$, $\abs{\gamma}=1.0346$, $\abs{\delta}=0.8861
$, $\abs{\alpha+\beta}=2.3130$ \& $\abs{\gamma+\delta}=1.8467$ & \begin{center}
$\abs{\alpha}>\abs{\gamma}$, $\abs{\beta}>\abs{\delta}$, $\abs{\alpha+\beta}>\abs{\gamma+\delta}$
\end{center}
&
\begin{center}
Lyapunav Exponent: $1.6225$, Fractal dimension:$1.4235$  \\ (Fig.14: Fractal-like)
\end{center}\\
\hline
\centering $\alpha=0.5078 + 0.5856i$, $\beta=0.7629 + 0.0830i$, $\gamma=0.6616 + 0.5170i
$, $\delta=0.1710 + 0.9386i$, $\abs{\alpha}=0.7751$, $\abs{\beta}=0.7674$, $\abs{\gamma}=0.8396$, $\abs{\delta}=0.9540$, $\abs{\alpha+\beta}=1.4359$ \& $\abs{\gamma+\delta}=1.6769$ & \begin{center}
$\abs{\alpha}<\abs{\gamma}$, $\abs{\beta}<\abs{\delta}$, $\abs{\alpha+\beta}<\abs{\gamma+\delta}$
\end{center}
&
\begin{center}
Lyapunav Exponent: $1.4872$ \\ (Fig.15: Fractal-unlike)
\end{center}\\
\hline
\end{tabular}
\caption{Chaotic trajectories of the equation Eq.(\ref{equation:total-equationB}) for different choice of parameters and ten set of initial values.}
\label{Table:}
\end{table}

\noindent
 The chaotic trajectory plots including corresponding complex plots of four examples whose parameters are given in Table $1$ starting from top row are given the following Fig.$12$, Fig.$13$, Fig.$14$ and Fig. $15$ respectively.

\begin{figure}[H]
      \centering

      \resizebox{14cm}{!}
      {
      \begin{tabular}{c c}
      \includegraphics [scale=6]{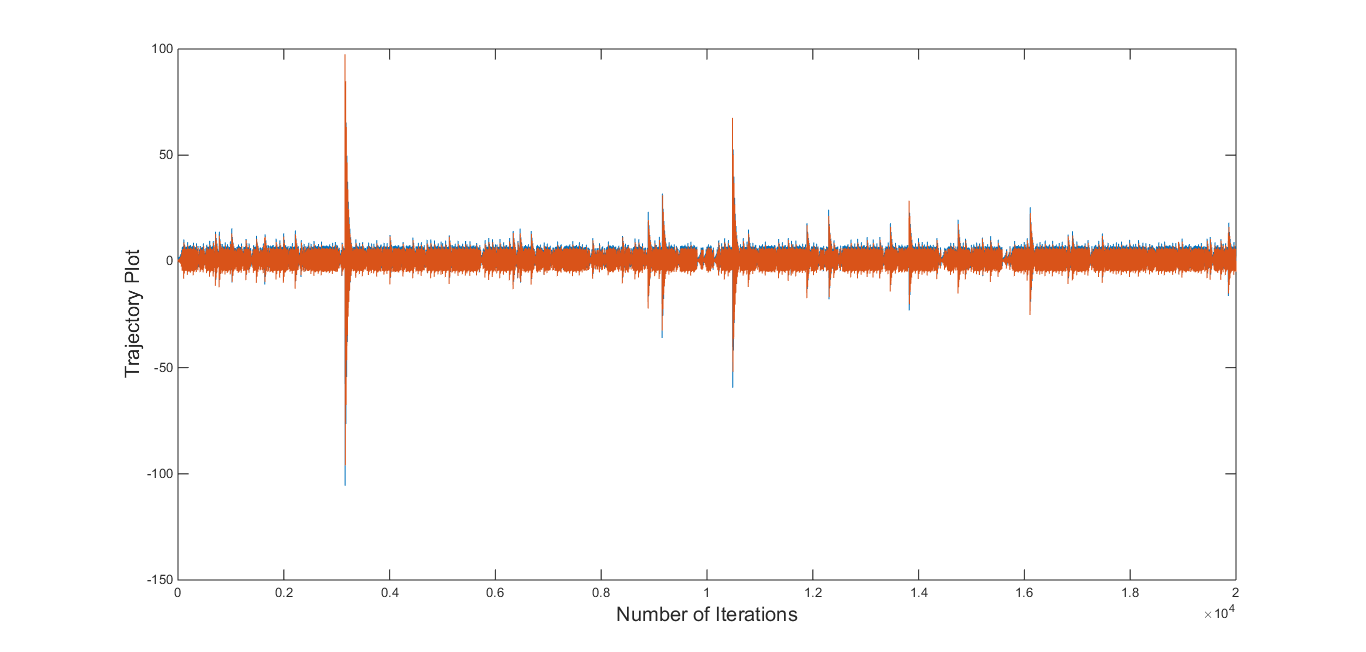}
      \includegraphics [scale=6]{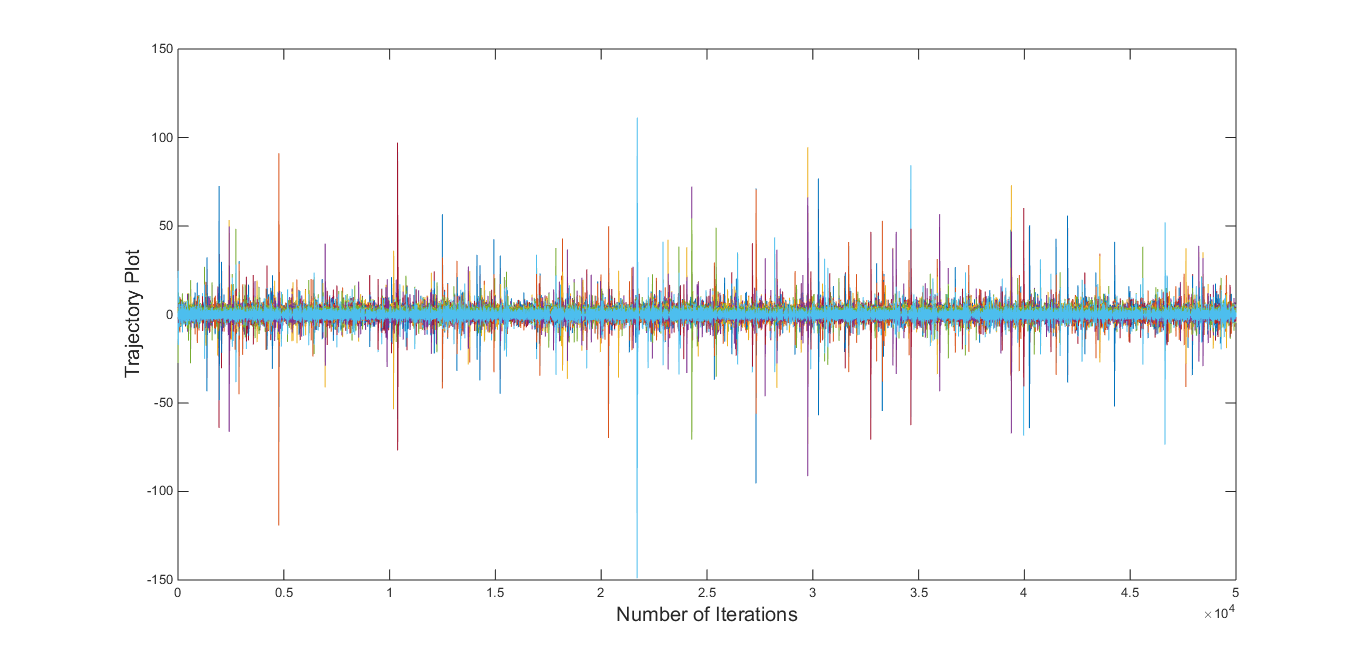}\\
      \includegraphics [scale=6]{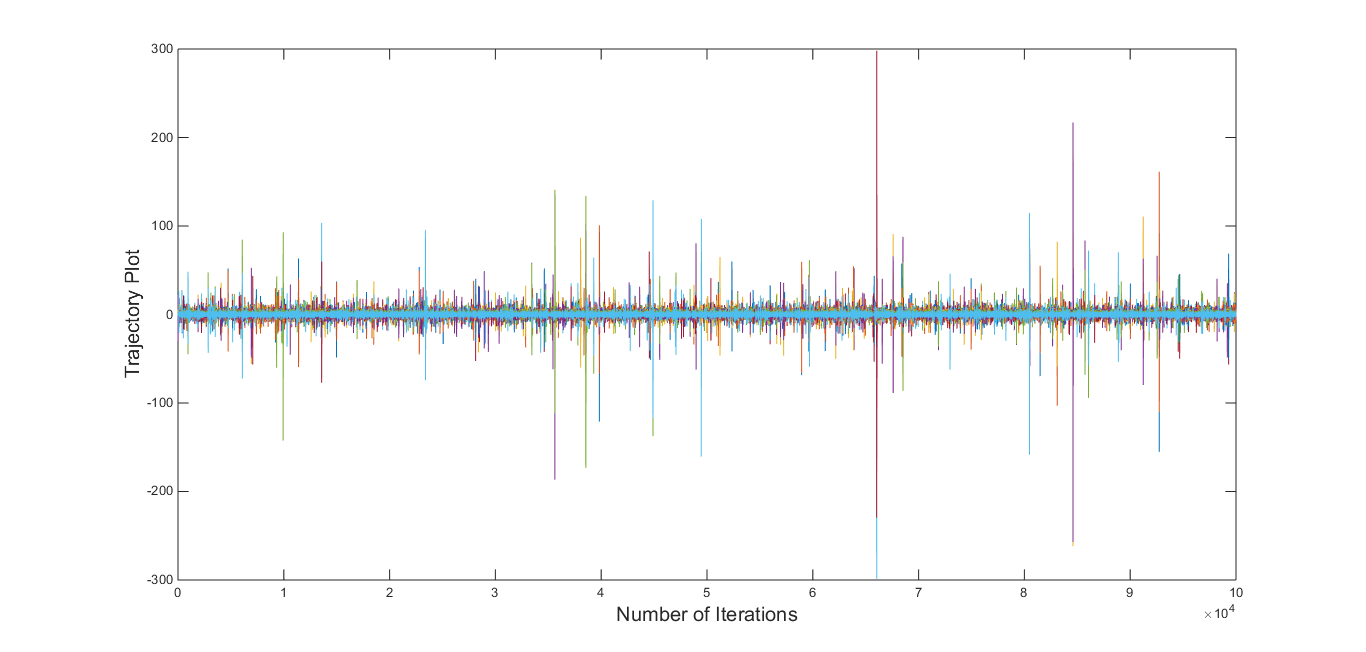}
      \includegraphics [scale=6]{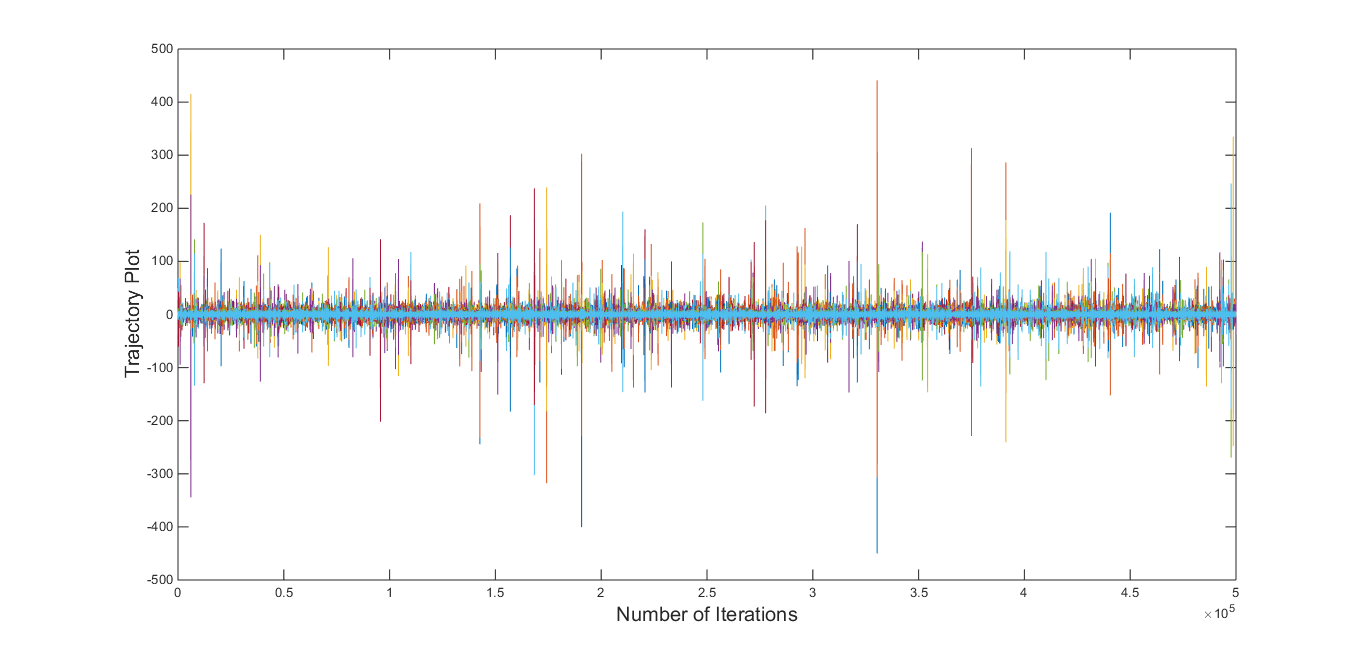}\\
      \end{tabular}
      }
\caption{Chaotic Trajectories of the equation Eq.(\ref{equation:total-equationB}) (Refer Row-1 of Table-1) with 10 set of initial values (10 different colors of trajectories).}
      \begin{center}

      \end{center}
      \end{figure}
\noindent

\begin{figure}[H]
      \centering

      \resizebox{14cm}{!}
      {
      \begin{tabular}{c c}
      \includegraphics [scale=6]{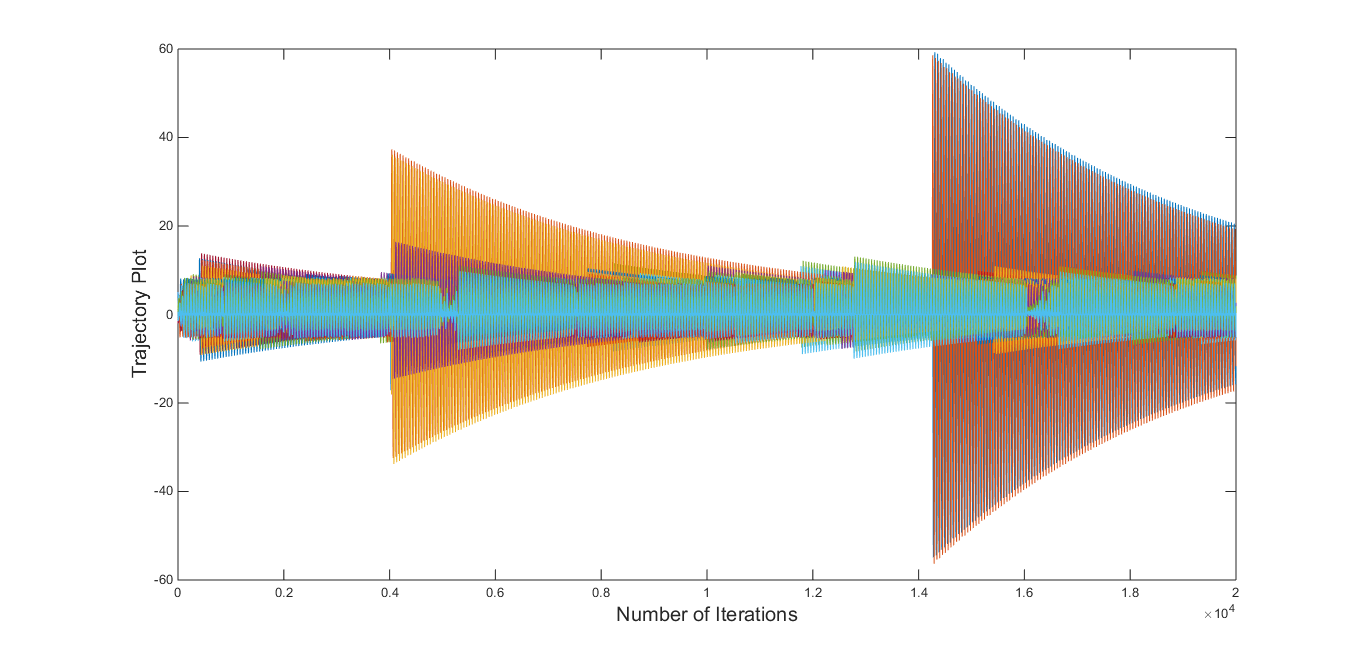}
      \includegraphics [scale=6]{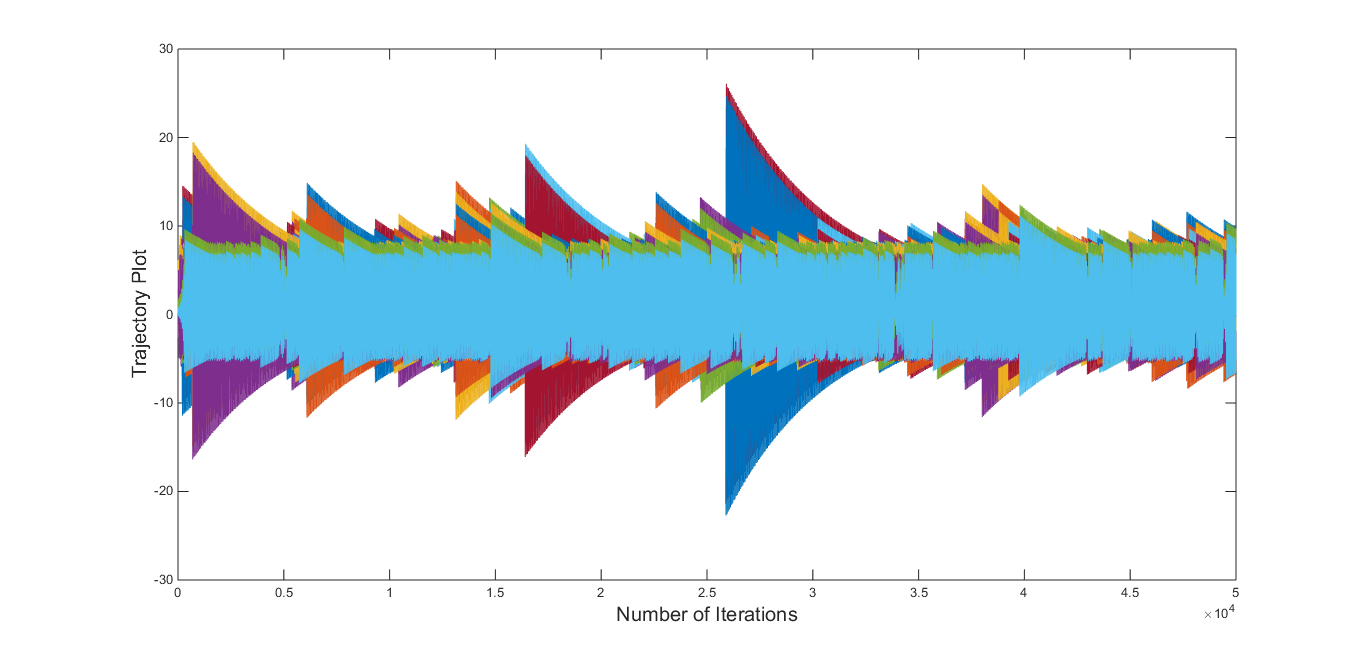}\\
      \includegraphics [scale=6]{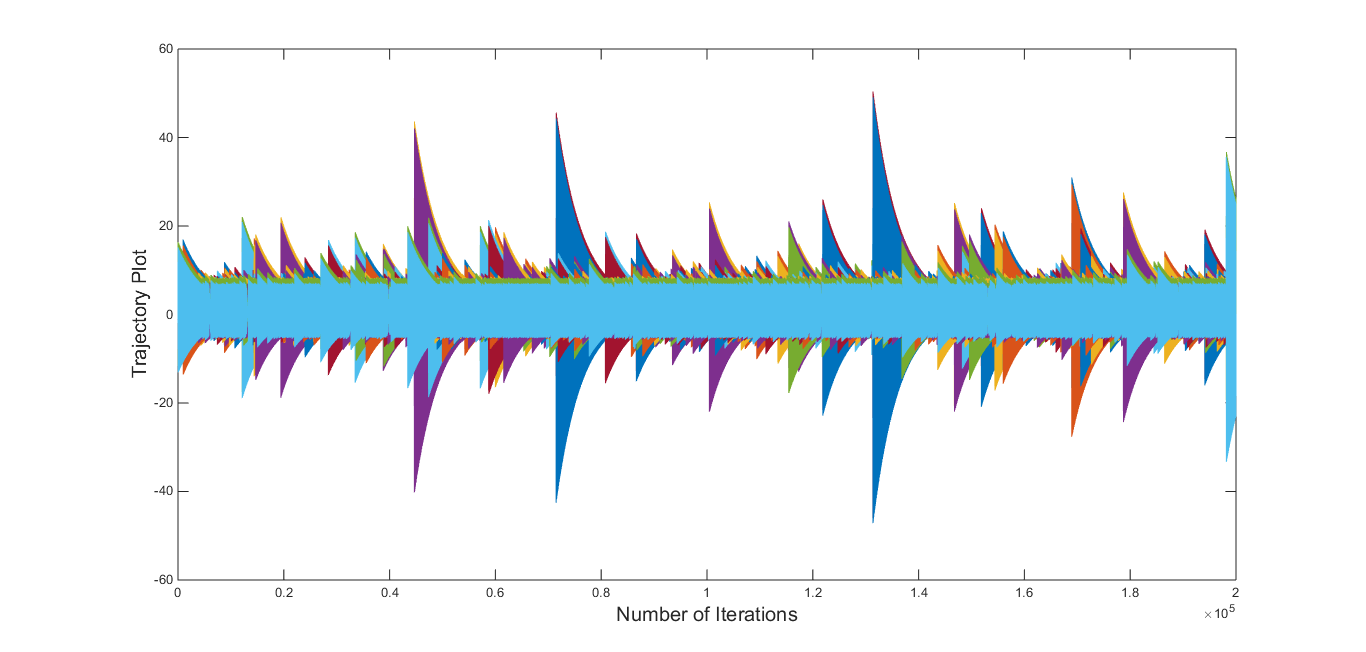}
      \includegraphics [scale=6]{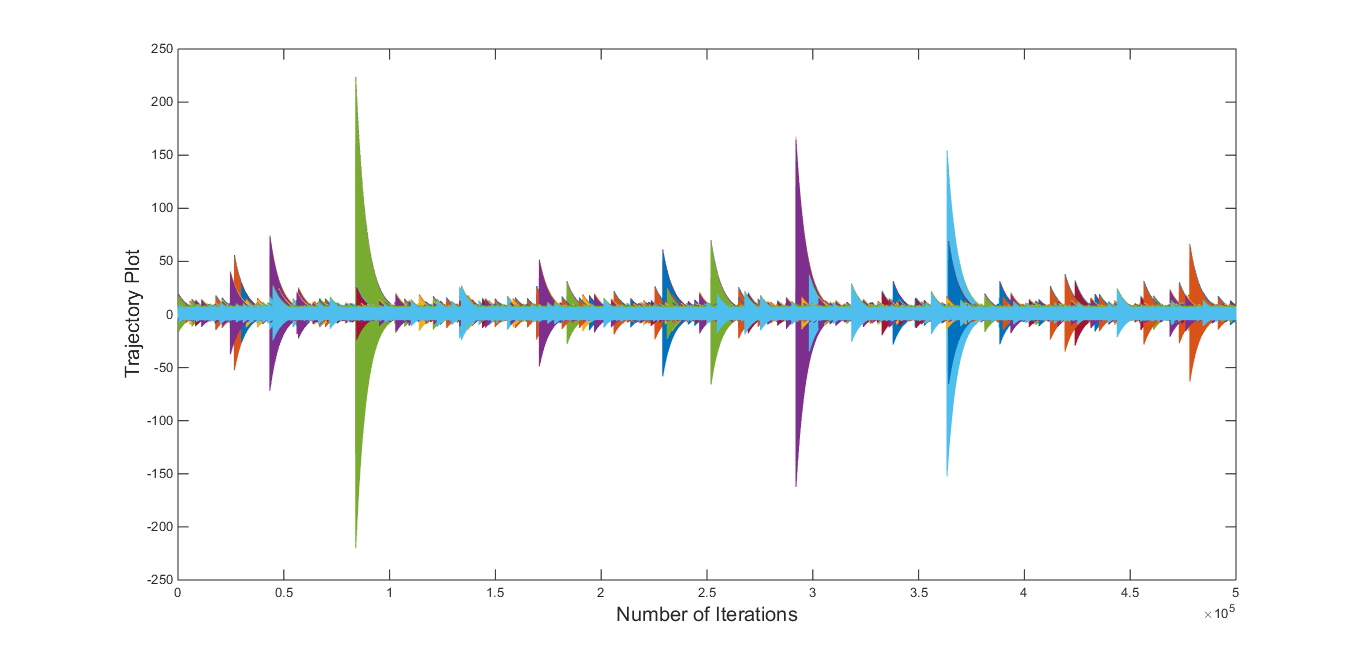}\\
      \end{tabular}
      }
\caption{Chaotic Trajectories of the equation Eq.(\ref{equation:total-equationB}) (Refer Row-2 of Table-1) with 10 set of initial values (10 different colors of trajectories).}
      \begin{center}

      \end{center}
      \end{figure}
\noindent

\begin{figure}[H]
      \centering

      \resizebox{14cm}{!}
      {
      \begin{tabular}{c c}
      \includegraphics [scale=6]{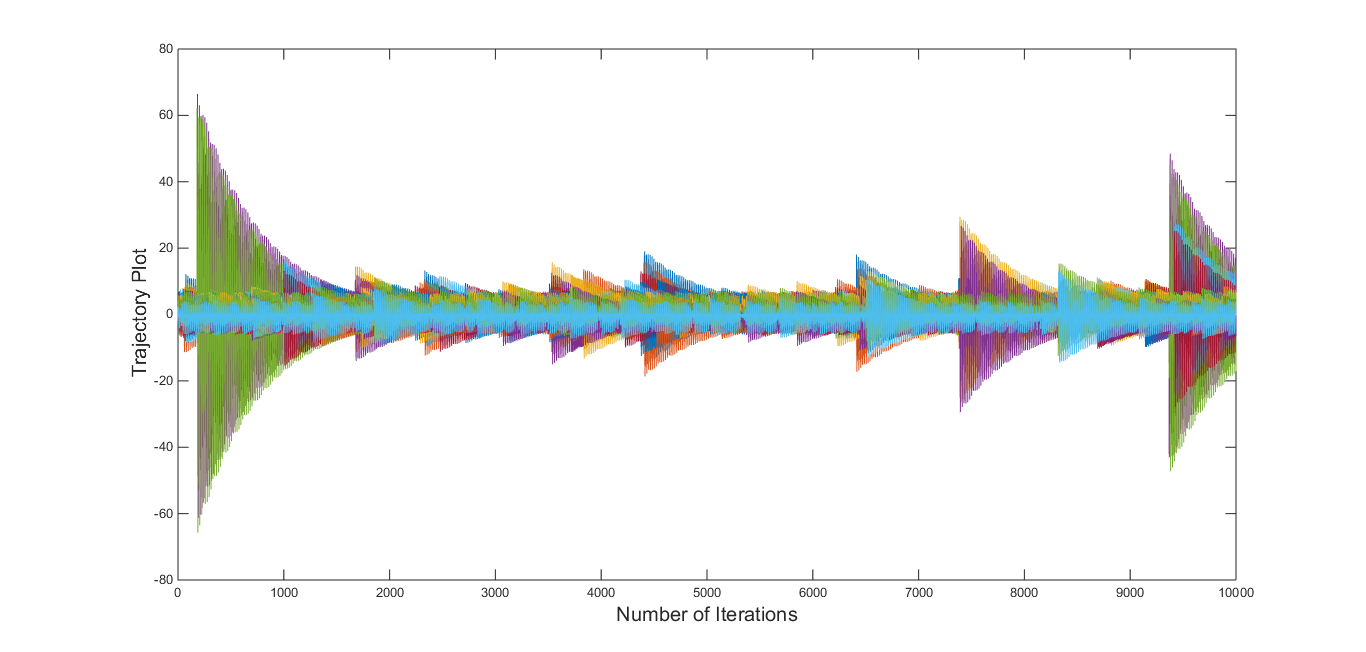}
      \includegraphics [scale=6]{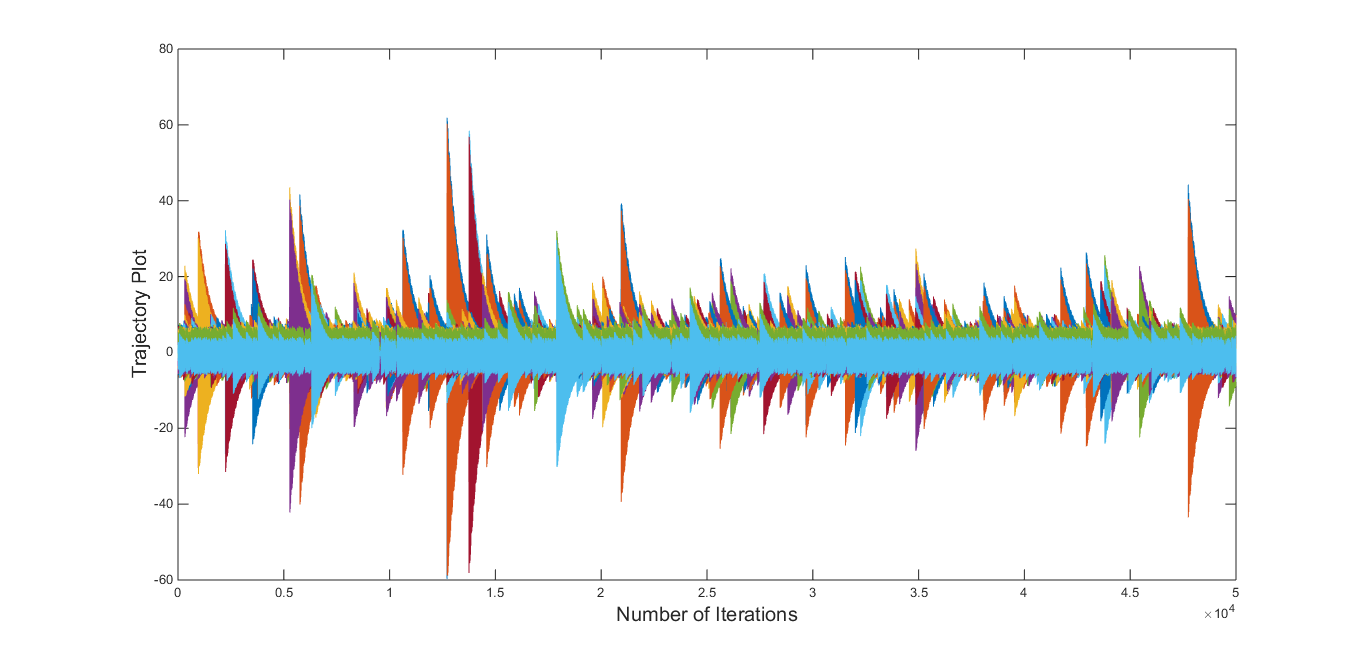}\\
      \includegraphics [scale=6]{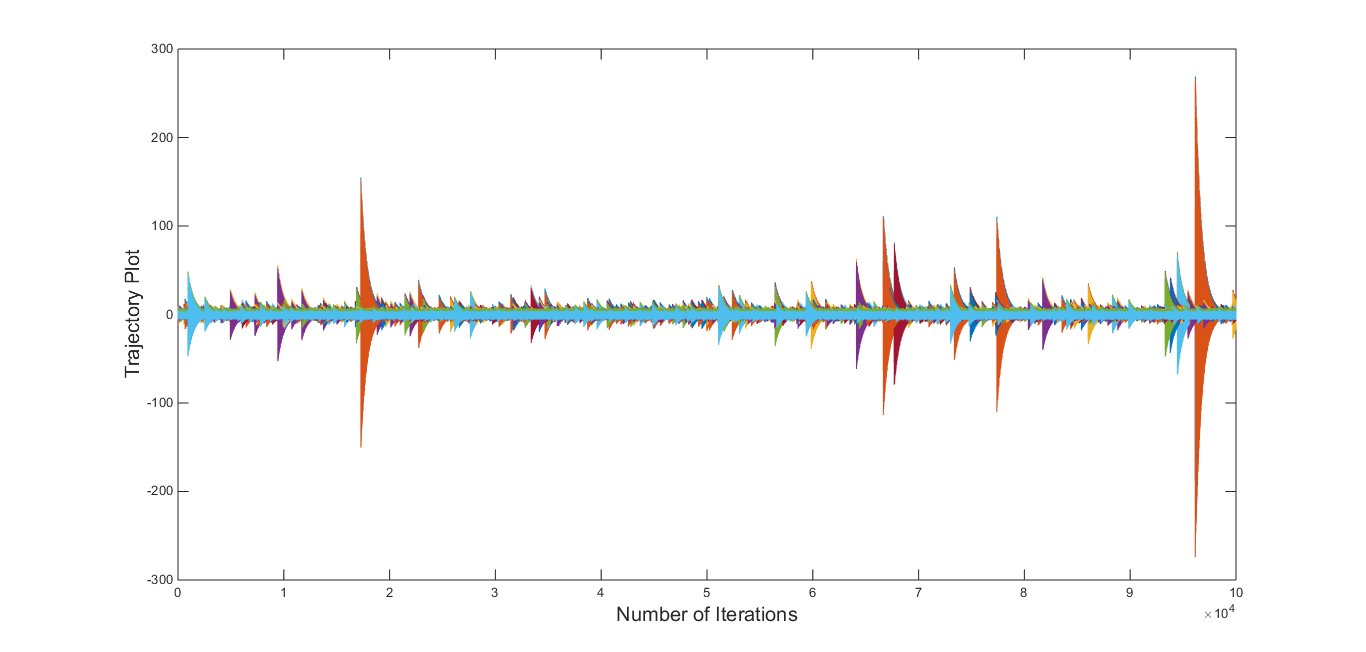}
      \includegraphics [scale=6]{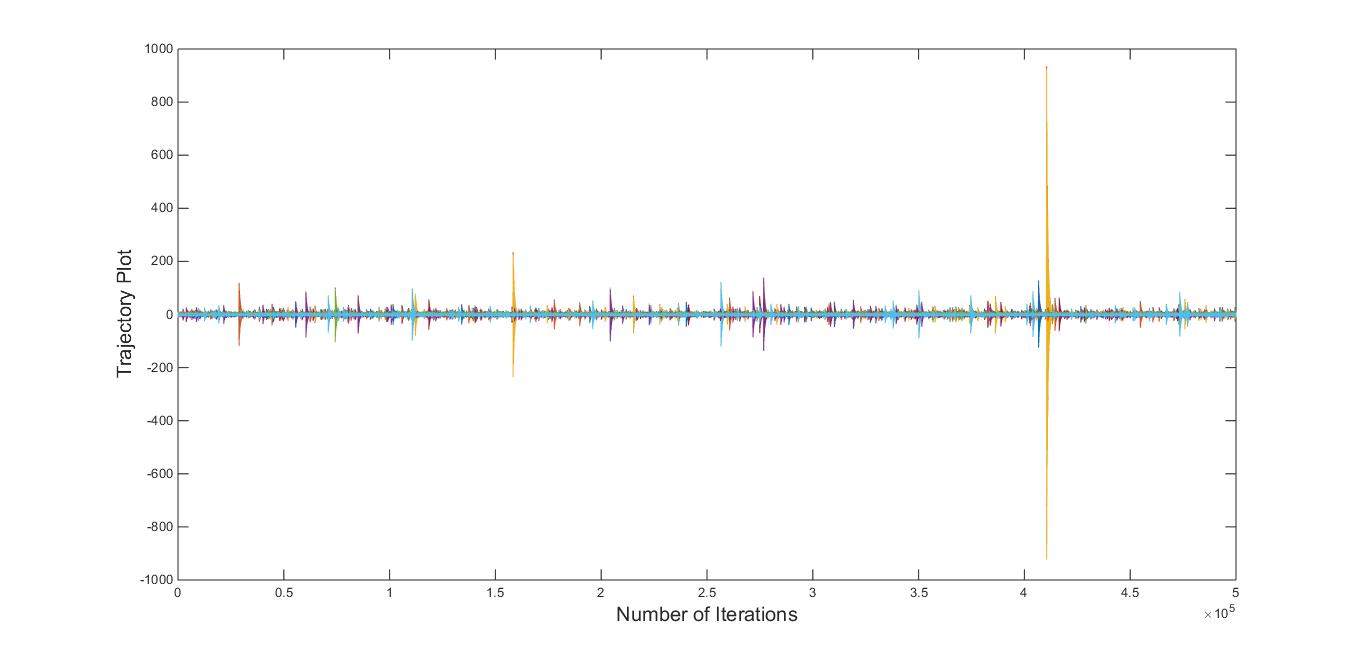}\\
      \end{tabular}
      }
\caption{Chaotic Trajectories of the equation Eq.(\ref{equation:total-equationB}) (Refer Row-3 of Table-1) with 10 set of initial values (10 different colors of trajectories).}
      \begin{center}

      \end{center}
      \end{figure}
\noindent

\begin{figure}[H]
      \centering

      \resizebox{12cm}{!}
      {
      \begin{tabular}{c c}
      \includegraphics [scale=6]{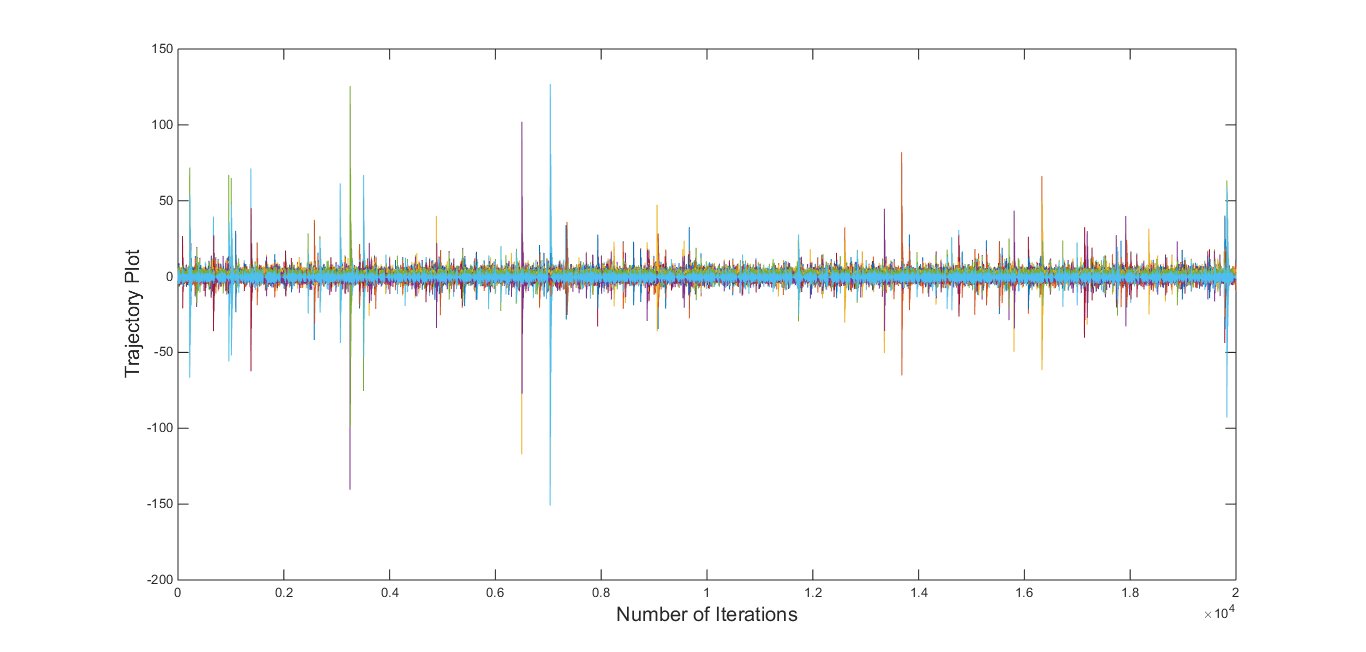}
      \includegraphics [scale=6]{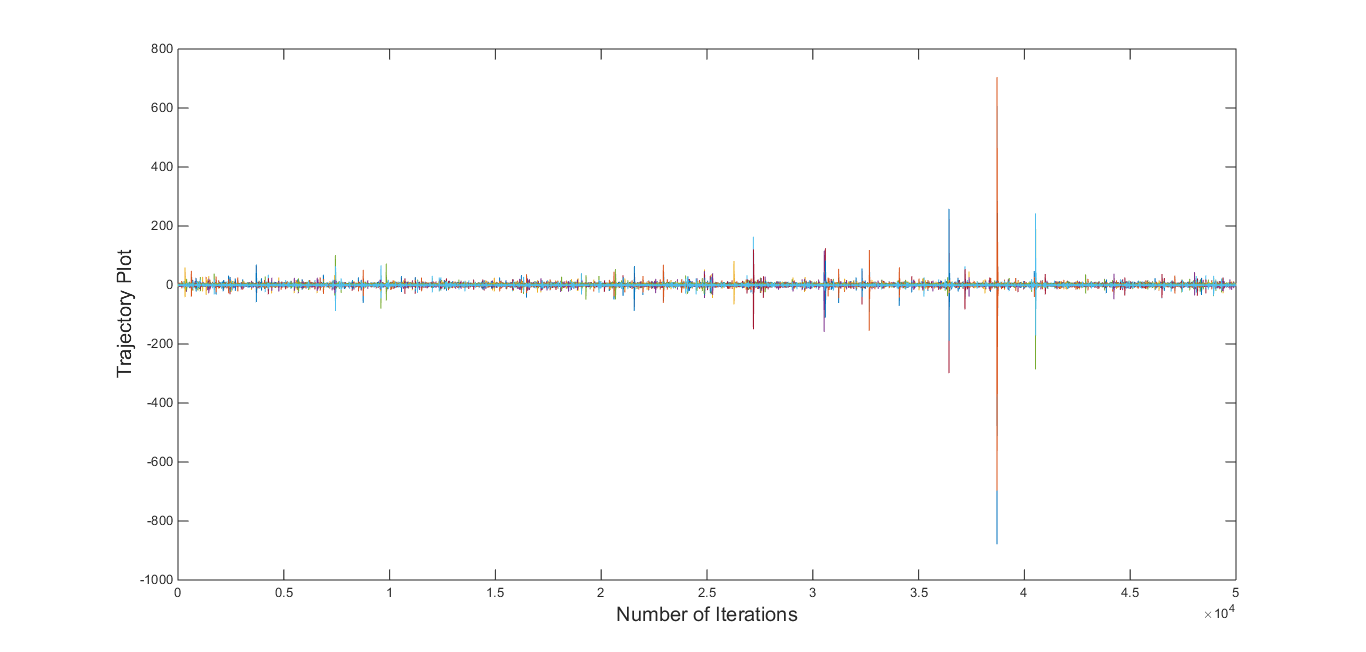}\\
      \includegraphics [scale=6]{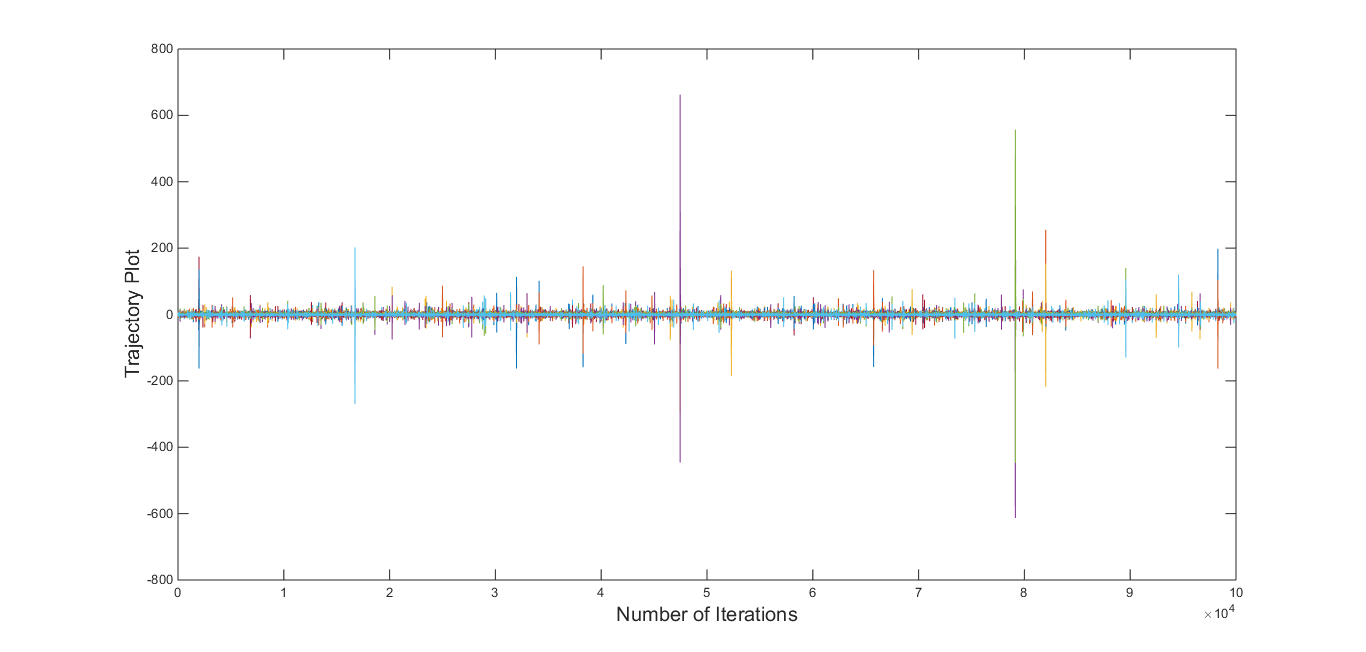}
      \includegraphics [scale=6]{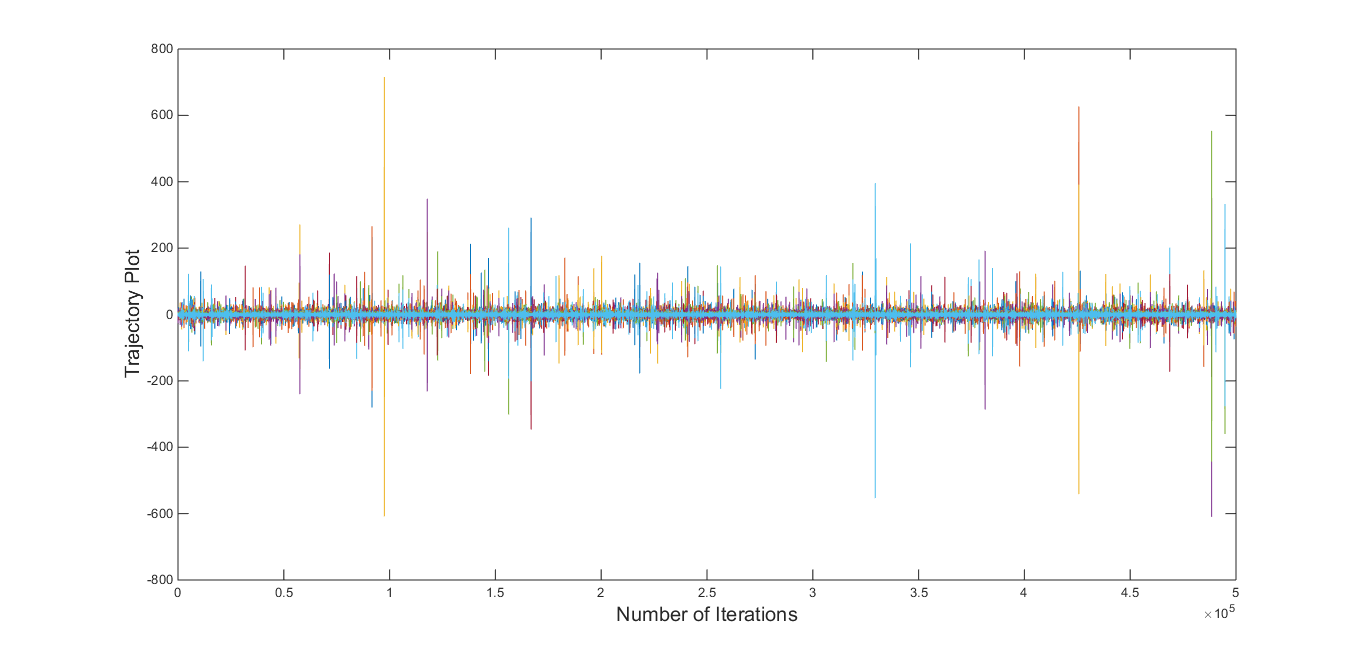}\\
      \end{tabular}
      }
\caption{Chaotic Trajectories of the equation Eq.(\ref{equation:total-equationB}) (Refer Row-4 of Table-1) with 10 set of initial values (10 different colors of trajectories).}
      \begin{center}

      \end{center}
      \end{figure}
\noindent

\noindent
It is needless to mention that the above fractal-like and fractal-unlike trajectories are chaotic as they have largest Lyapunav exponent as positive which are shown in the Table $1$. The fractal-like trajectories are proven to be fractals by the fractal dimension which are fractional. \\

\noindent
The fractal-unlike chaotic trajectories are plotted for $10$ different initial values in the Fig. $12$, and Fig.$15$. The fractal-like chaotic trajectories are plotted for ten initial values in the Fig.$13$ and Fig. $14$. \\ It is observed that in the in the case of fractal-like trajectories, the parameters follow the following conditions: $\abs{\alpha}<\abs{\gamma}$, $\abs{\beta}<\abs{\delta}$, $\abs{\alpha+\beta}<\abs{\gamma+\delta}$ and in fractal unlike trajectories, the parameters follow the conditions: $\abs{\alpha}>\abs{\gamma}$, $\abs{\beta}>\abs{\delta}$, $\abs{\alpha+\beta}>\abs{\gamma+\delta}$. We also perform a set of 500 such chaotic trajectories are studied and it is found that the parameters corresponding to the fractal-like chaotic trajectories are following the condition: $\abs{\alpha}<\abs{\gamma}$, $\abs{\beta}<\abs{\delta}$, $\abs{\alpha+\beta}<\abs{\gamma+\delta}$ and the same for fractal-unlike chaotic trajectories are following the condition: $\abs{\alpha}>\abs{\gamma}$, $\abs{\beta}>\abs{\delta}$, $\abs{\alpha+\beta}>\abs{\gamma+\delta}$ (data not shown). From these two computational observations, the following conjecture has been made.

\begin{conjecture}
The chaotic trajectories Eq.(\ref{equation:total-equationB}) are fractal-like only if the parameters satisfy the conditions: $\abs{\alpha}<\abs{\gamma}$, $\abs{\beta}<\abs{\delta}$, $\abs{\alpha+\beta}<\abs{\gamma+\delta}$ and fractal-unlike only if $\abs{\alpha}>\abs{\gamma}$, $\abs{\beta}>\abs{\delta}$, $\abs{\alpha+\beta}>\abs{\gamma+\delta}$

\end{conjecture}

\section{A Few Special Cases of Parameters}
Here we shall explore a few case study by considering some restrictions on the parameters $\alpha$, $\beta$, $\gamma$ and $\delta$ as follows. \\

\subsection{$\alpha=\beta, \gamma=\delta$}
Consider $\alpha=\beta$ and $\gamma=\delta$ in the dynamical system Eq.(\ref{equation:total-equationB}) and then it becomes

\begin{equation}
\displaystyle{z_{n+1}=f_{\alpha,\gamma}(z_n)=\frac{\alpha (z_n + 1)}{\gamma (z_n^2 + z_n)}}=\frac{\alpha}{\gamma z_n}
\label{equation:total-equationC}
\end{equation}%
\noindent
The Eq.(\ref{equation:total-equationC}) has two fixed points $z_{\pm}^*=$ $\pm \sqrt{\frac{\alpha}{\gamma}}$ and both of which are neither stable nor unstable since the $\abs{f'(z_{\pm}^*)}=1$.\\
\noindent
The Eq.(\ref{equation:total-equationC}) has periodic solutions of period $2$ only and they are $\phi= \frac{\alpha}{\gamma \psi}$ and $\psi=\frac{\alpha}{\gamma \phi}$. It is noted that there is no nontrivial periodic points of higher periods. For $50$ different initial values the trajectories are plotted in the following Fig. $16$.

\begin{figure}[H]
      \centering

      \resizebox{13cm}{!}
      {
      \begin{tabular}{c }
      \includegraphics [scale=7]{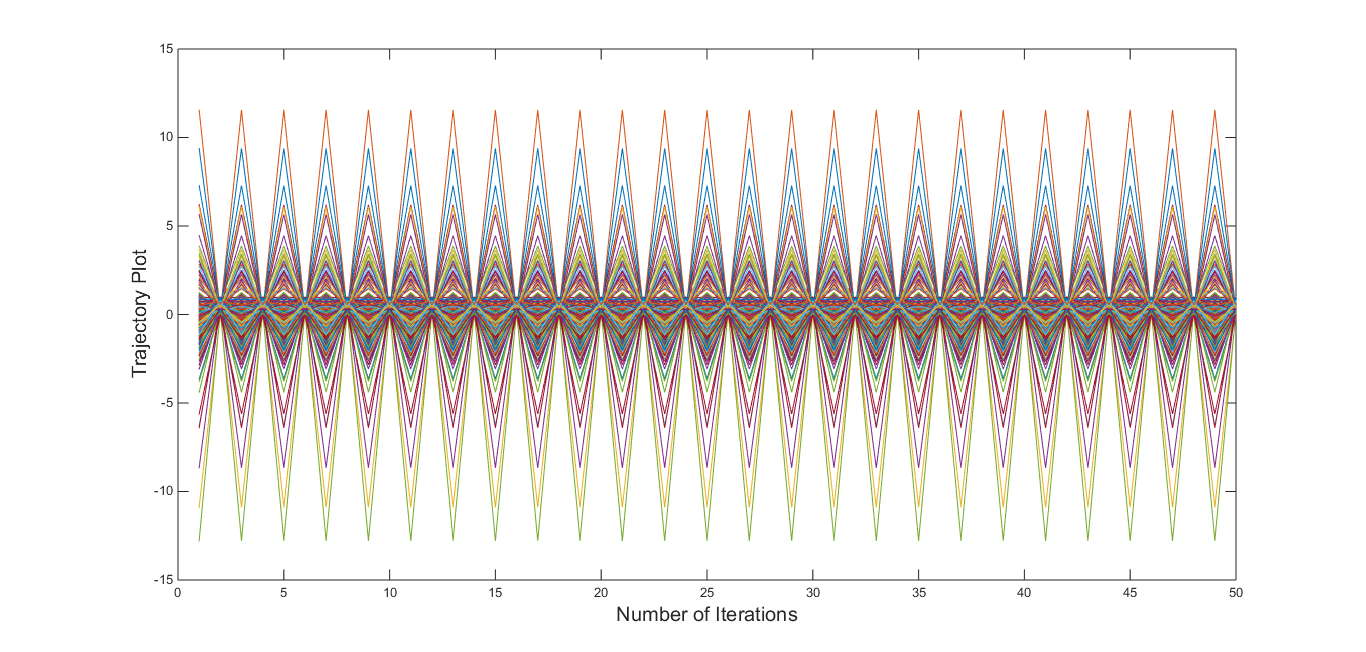}
      \end{tabular}
      }
\caption{50 Period 2 Trajectories of the equation Eq.(\ref{equation:total-equationC}).}
      \begin{center}

      \end{center}
      \end{figure}
\noindent

\subsection{$\alpha=-\beta, \gamma=\delta$}
Consider $\alpha=-\beta$ and $\gamma=\delta$ in the dynamical system Eq.(\ref{equation:total-equationB}) and then it becomes

\begin{equation}
\displaystyle{z_{n+1}=f_{\alpha,\gamma}(z_n)=\frac{\alpha (z_n - 1)}{\gamma (z_n^2 + z_n)}}
\label{equation:total-equationD}
\end{equation}%

\noindent
For arbitrary any initial value, the trajectory is convergent and converge to $\infty$ which is a member of the extended complex plane $\mathbb{C}^*$ which is shown in the Fig.$17$.

\begin{figure}[H]
      \centering

      \resizebox{14cm}{!}
      {
      \begin{tabular}{c c}
      \includegraphics [scale=7]{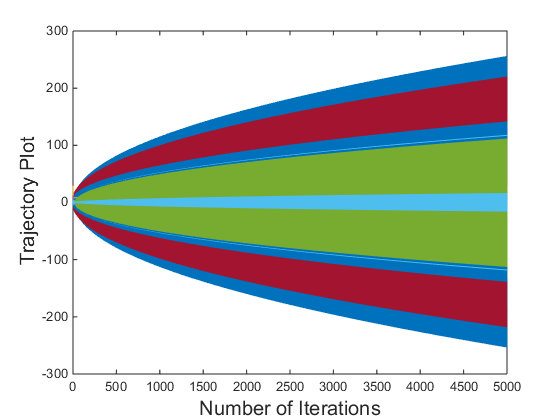}
      \includegraphics [scale=7]{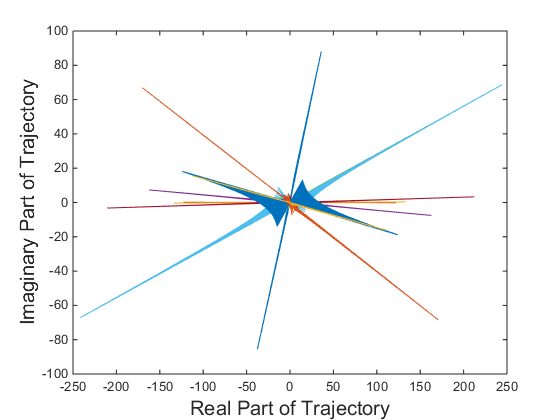}
\end{tabular}
      }
\caption{Trajectory plots of the equation Eq.(\ref{equation:total-equationD}).}
      \begin{center}

      \end{center}
      \end{figure}
\noindent

\subsection{$\gamma=\alpha, \delta=\beta$}
Consider $\gamma=\alpha$ and $\delta=\beta$ in the dynamical system Eq.(\ref{equation:total-equationB}) and then it becomes

\begin{equation}
\displaystyle{z_{n+1}=f_{\alpha,\beta}(z_n)=\frac{\alpha z_n +\beta}{\alpha z_n^2 + \beta z_n }}
\label{equation:total-equationE}
\end{equation}%

\noindent
The Eq.(\ref{equation:total-equationE}) has two fixed points $z_{\pm}^*=$ $\pm 1$ and both of which is neither stable nor unstable since the $\abs{f'(z_{\pm}^*)}=1$ but the Eq.(\ref{equation:total-equationE}) possess the period $2$ solution which are reciprocal to each other. For any initial values (except $-\frac{\beta}{\alpha}$) all solutions of the Eq.(\ref{equation:total-equationE}) are convergent and converge to the period two solution.

\subsection{$\gamma=\beta, \delta=\alpha$}

Consider $\gamma=\beta$ and $\delta=\alpha$ in the dynamical system Eq.(\ref{equation:total-equationB}) and then it becomes

\begin{equation}
\displaystyle{z_{n+1}=f_{\alpha,\beta}(z_n)=\frac{\alpha z_n +\beta}{\beta z_n^2 + \alpha z_n }}
\label{equation:total-equationF}
\end{equation}%

\noindent
The Eq.(\ref{equation:total-equationF}) has three fixed points $z^*=1$, $z^*=\frac{-0.5 \alpha -0.5 \beta -0.5 \sqrt{\alpha ^2+2 \alpha  \beta -3\beta ^2}}{\beta }$ and $z^*=\frac{0.5 \left(- \alpha -\beta +\sqrt{\alpha ^2+2\alpha  \beta -3\beta ^2}\right)}{\beta }$.The local asymptotically stability of the fixed point $1$ depends on $\abs{f'(z^*)}=\frac{2 \beta }{\alpha +\beta }$. So it turns out that if $2\abs{\beta}<\abs{\alpha+\beta}<\abs{\alpha}+\abs{\beta}$ i.e. $\abs{\beta}<\abs{\alpha}$ then the fixed point $1$ is a sink and otherwise it is a source. \\
If $\abs{-2-\frac{\alpha }{\beta }}<1$ i.e. $\abs{\alpha+2\beta}<\abs{\beta}$ then the fixed points $z^*=\frac{-0.5 \alpha -0.5 \beta -0.5 \sqrt{\alpha ^2+2 \alpha  \beta -3\beta ^2}}{\beta }$ and $z^*=\frac{0.5 \left(- \alpha -\beta +\sqrt{\alpha ^2+2\alpha  \beta -3\beta ^2}\right)}{\beta }$ would be a sink. \\
\noindent
Beside the regular dynamics the Eq.(\ref{equation:total-equationF}) possess chaotic solutions. Here is an example of chaotic solutions for $10$ different initial values. Here $\alpha=\delta=28+68i$ and $\beta=\gamma=66+17i$ and $\abs{\alpha}=73.5391$, $\abs{\beta}=68.1542$ and $\abs{\alpha+\beta}=\abs{\gamma+\delta}=126.7320$ i.e. $\abs{\alpha}>\abs{\gamma}$ $\abs{\beta}<\abs{\delta}$ and $\abs{\alpha+\beta}=\abs{\gamma+\delta}$. The chaotic solution is plotted in the Fig. $18$.

\begin{figure}[H]
      \centering

      \resizebox{14cm}{!}
      {
      \begin{tabular}{c c}
      \includegraphics [scale=7]{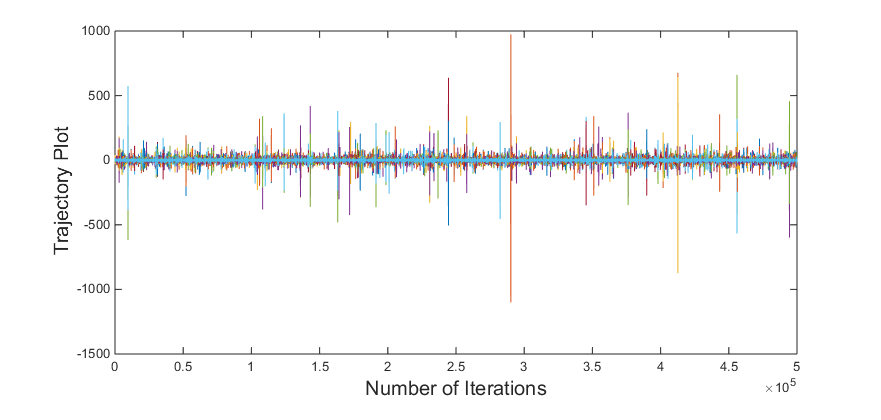}
      \includegraphics [scale=7]{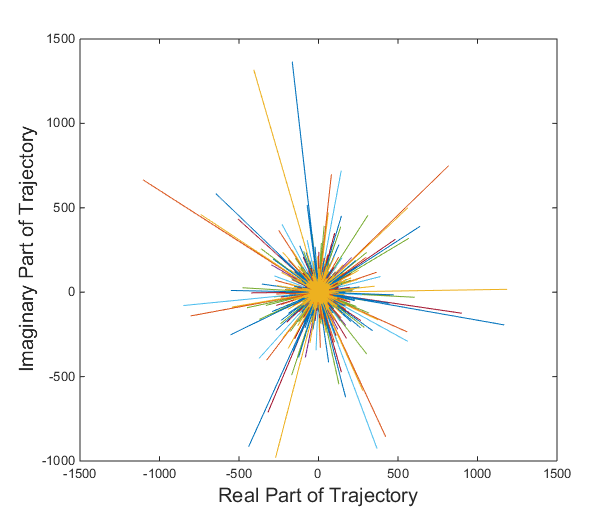}
\end{tabular}
      }
\caption{Chaotic Trajectory plots of the equation Eq.(\ref{equation:total-equationF}).}
      \begin{center}

      \end{center}
      \end{figure}

\section{Future Endeavours}

It is seen that the dynamical system Eq.(\ref{equation:total-equationB}) exhibits a phenomenal jump in terms of dynamical behaviour from real line to complex plane. In continuation of the present work the study of the dynamical system with time delay $l$, ${z_{n+1}=\frac{\alpha_n z_{n-l}+\beta_n}{\gamma_n z_{n-l}^2 + \delta_n z_{n-l}}}$ where $\alpha_n$, $\beta_n$, $\gamma_n$ and $\delta_n$ are all convergent sequence of complex numbers and converges to $\alpha$, $\beta$, $\gamma$ and $\delta$ respectively would indeed be an interesting problem and which we would like to pursue further.



\end{document}